\documentclass[review, 3p, compress]{elsarticle}
\usepackage{hyperref}
\usepackage{amsfonts}
\usepackage{amssymb}
\usepackage{amsmath}
\usepackage{cases}
\usepackage{graphicx}
\usepackage{rotating}
\usepackage{float}
\usepackage{booktabs}
\usepackage{multirow}
\usepackage{multicol}
\usepackage{epstopdf}
\usepackage{subfigure}
\usepackage{color}
\usepackage{setspace}
\usepackage{bm}
\usepackage{doi}
\usepackage{algorithm}
\usepackage{algorithmic}
\usepackage{natbib}
\usepackage{geometry}
\makeatletter
\@addtoreset{equation}{section}
\makeatother

\newtheorem{theorem}{Theorem}[section]
\newtheorem{lemma}{Lemma}[section]

\begin{document}
\begin{frontmatter}
\title{A second-order accurate implicit difference scheme for time fractional
reaction-diffusion equation with variable coefficients and time drift term}

\author[address1]{Yong-Liang Zhao\corref{correspondingauthor}}
\ead{uestc\_ylzhao@sina.com}

\author[address1]{Pei-Yong Zhu}
\ead{zpy6940@uestc.edu.cn}

\author[address2]{Xian-Ming Gu\corref{correspondingauthor}}
\cortext[correspondingauthor]{Corresponding author}
\ead{guxianming@live.cn}

\author[address1]{Xi-Le Zhao}
\ead{xlzhao122003@163.com}

\address[address1]{School of Mathematical Sciences, \\
	University of Electronic Science and Technology of China, \\
	Chengdu, Sichuan 611731, P.R. China}
\address[address2]{School of Economic Mathematics/Institute of Mathematics,\\
	Southwestern University of Finance and Economics,\\
	Chengdu, Sichuan 611130, P.R. China}
\begin{abstract}
An implicit finite difference scheme based on the $L2$-$1_{\sigma}$ formula is presented for
a class of one-dimensional time fractional reaction-diffusion equations with variable coefficients
and time drift term. The unconditional stability and convergence of this scheme
are proved rigorously by the discrete energy method, and the optimal convergence order in the $L_2$-norm is
$\mathcal{O}(\tau^2 + h^2)$ with time step $\tau$ and mesh size $h$.
Then, the same measure is exploited to solve the two-dimensional case of this problem
and a rigorous theoretical analysis of the stability and convergence is carried out.
Several numerical simulations are provided to show the efficiency and accuracy of our proposed schemes
and in the last numerical experiment of this work, three preconditioned iterative methods are employed for solving
the linear system of the two-dimensional case.
\end{abstract}

\end{frontmatter}

%

\section{Introduction}
\label{sec1}
In the past decades, fractional calculus has growing interest been paid in modelling applications,
including the spread of HIV infection of CD4$+$ T-cells \cite{ding2009fractional}, entropy \cite{povstenko2015generalized},
hydrology \cite{benson2000application}, soft tissues such as mitral valve in the human heart \cite{shen2013fractional},
anomalous diffusion in transport dynamics of complex systems \cite{klages2008anomalous}, engineering and physics.
Many other examples can be found in Refs. \cite{Alfonso2014Fractional, Jiang2012Analytical, Qin2017308}.
In these models, the fractional diffusion equation (FDE)  has been studied by many researchers,
see \cite{luo2016quadratic, gu2015strang, Gu2017fast, Gao2016Two, hao2015high, cao2015compact,
du2016lubich, li20172d, zhou2018galerkin, li2016multifde} and references therein. 

Since the solution of fractional operator at a given point depends on the solution behavior on the entire domain,
i.e., the fractional operators are nonlocal,
fractional diffusion equations (FDEs) tend to be more appropriate for the description of various materials
and processes with memory and hereditary properties than the normal integer-order counterparts.
At the same time, the nonlocal nature of fractional operators has an inherent challenge when facing FDEs,
namely the analytical solutions of FDEs are difficult to obtain,
except for some special cases \cite{podlubny1998fractional}.
For this reason, the proposal and study of numerical methods that are efficient,
accurate and easy to implement, are quite essential in obtaining the approximate solutions of FDEs.
Without doubt, it is worth noting that there still are a few effective numerous analytical methods,
for instance the Laplace transform method, the Fourier transform method and Adomian decomposition method.
Up to now abundant numerical methods have been proposed for solving the FDEs,
e.g., finite difference method \cite{Gu2017fast, Gao2016Two, hao2015high, cao2015compact,
Gu2014On, zhao2018delay}, finite element method \cite{Li2016Galerkin, Li2018fast, li2018fem},
collocation method \cite{luo2016quadratic}, meshless method \cite{dehghan2015error}
and spectral method \cite{Mao2016Efficient}.
Among them, the finite difference scheme is one of the most popular numerical schemes
employed to solve space and/or time FDEs, and we only mention some works in the next.

For the space FDEs, Meerschaert and Tadjeran \cite{Meerschaert2004Finite}
used the implicit Euler method based on the standard Gr\"{u}nwald-Letnikov formula
to discrete space-fractional advection-dispersion equation with first order accuracy,
but the obtained implicit difference scheme (IDS) is unstable.
To overcome this problem, Meerschaert and Tadjeran \cite{Meerschaert2004Finite}
first proposed the shifted Gr\"{u}nwald-Letnikov formula, which is unconditionally stable.
After their study, second-order approximations to space FDEs have been investigated,
Sousa and Li \cite{Sousa2015A} derived an unconditionally stable weighted average finite difference formula
for one-dimensional FDE with convergence $\mathcal{O}(\tau + h^2)$
where $\tau$ and $h$ are time step and mesh size, respectively.
Tian et al. \cite{tian2015class} proposed a class of second-order approximations,
which are termed as the weighted and shifted Gr\"{u}nwald difference (abbreviated as WSGD) operators,
to solve the two-sided one-dimensional space FDE numerically.
As expected, the convergence rate of their IDS is $\mathcal{O}(\tau^2 + h^2)$
by combining the Crank-Nicolson method (C-N).
Adopting the same idea and utilizing the quasi-compact numerical technique,
Zhou et al. \cite{zhou2013quasi} obtained a numerical approximate scheme
with convergence $\mathcal{O}(\tau^2 + h^3)$.
Subsequently, Hao et al. \cite{hao2015fourth}
applied a new fourth-order difference approximation,
which was derived by using the weighted average of the shifted Gr\"{u}nwald formulae
and combining with the compact numerical technique, to solve the two-sided one-dimensional space FDE.
They proved that the proposed quasi-compact difference scheme is unconditionally stable
and convergent in $L_2$-norm with the optimal order $\mathcal{O}(\tau^2 + h^4)$.
On the other hand, for the time FDEs, many early researches \cite{murillo2011explicit, gao2011compact, gao2012finite}
employed the $L1$ formula to obtain their difference schemes.
Then, Gao et al. \cite{gao2014new} applied
their new fractional numerical differentiation formula (called the $L1$-$2$ formula)
to solve the time-fractional sub-diffusion equations
with accuracy $\mathcal{O}(\tau^{3 - \alpha} + h^2)~(0 < \alpha < 1)$.
Alikhanov \cite{Alikhanov2015A} proposed a modified scheme,
which is of second order accuracy. The stability of his scheme was then proved
and numerical evidence has shown that this scheme for the $\alpha$-order Caputo fractional derivative
is of  second order accuracy.
Later, based on this modified scheme,
Yan et al. \cite{yan2017fastcaputo} designed
a fast high-order accurate numerical scheme (named $FL2$-$1_\sigma$)
to speed up the evaluation of the Caputo fractional derivative.
This scheme efficiently reduces the computational storage and cost for solving the time FDEs.
Although there are many studies on the space/time FDEs, numerical studies on space-time FDEs are still not extensive,
the readers are suggested to see \cite{Gu2017fast, Liu2007Stability, Shen2011Numerical, lin2018separable}
and references therein.

In this manuscript, a second-order IDS is concerned
for solving the initial-boundary value problem
of the one-dimensional (1D) time fractional reaction-diffusion equation (TFRDE)
with variable coefficients and time drift term:
\begin{equation}
\begin{cases}
\frac{\partial u(x,t)}{\partial t} + D_{0,t}^{\alpha} u(x,t) = \mathcal{L} u(x,t) + f(x,t),
& 0 \leq x \leq L, ~~  0 \leq t \leq T, \\
u(x,0) = u_{0}(x),  & 0 \leq x \leq L,\\
u(0,t) = \phi_1(t), ~~ u(L,t) = \phi_2(t),  & 0 \leq t \leq T,
\end{cases}
\label{eq1.1}
\end{equation}
where
\begin{equation*}
\mathcal{L} u(x,t) = \frac{\partial}{\partial x} \Big( k(x,t) \frac{\partial u(x,t)}{\partial x} \Big) - q(x,t) u(x,t),
\end{equation*}
$k(x,t) \geq C > 0$, $q(x,t) \geq 0$ and $f(x,t)$ are sufficiently smooth functions.
Moreover, the time fractional derivative in\/ \eqref{eq1.1} is
the Caputo fractional derivative \cite{podlubny1998fractional} with order $\alpha \in (0,1]$ 
denoted by
\begin{equation}
D_{0,t}^{\alpha} u(x,t) = \frac{1}{\Gamma(1 - \alpha)}
\int_{0}^{t} \frac{\partial u(x,\xi)}{\partial \xi} \frac{d \xi}{(t - \xi)^{\alpha}}.
\label{eq1.2}
\end{equation}
The time drift term $\frac{\partial u(x,t)}{\partial t}$ is added to describe the motion time,
and this helps to distinguish the status of particles conveniently.
In particular, when $k(x,t) \equiv k$ is a constant and $q(x,t) = 0$,
Eq.\/ \eqref{eq1.1} reduces to a special time fractional
mobile/immobile transport model introduced in \cite{Liu2014A, Schumer2003Fractal}.

The rest of the paper is organized as follows: For clarity of presentation,
in the next section the full discretization of  Eq. \eqref{eq1.1} is introduced first,
then the stability analysis of the discrete scheme is carried out, and an error estimate shows that
the discrete scheme accuracy is of  $\mathcal{O}(\tau^2 + h^2)$.
In Section 3, we extend the TFRDE to two dimension, and the unconditionally stable and convergence are also proved.
Numerical examples are presented in Section 4 to illustrate the effectiveness of our proposed methods.
At last, some conclusions are drawn in Section 5.

\section{An implicit difference scheme for TFRDE}
\label{sec2}
In this section, an IDS is derived to discretize
the TFRDE defined in\/ \eqref{eq1.1}, and the stability and
error estimate of the IDS are analyzed in detail.

\subsection{Derivation of the second-order difference scheme}
\label{sec2.1}
To establish the numerical simulation scheme, we first discrete the solution region and
let the mesh $\bar{\omega}_{h \tau} = \bar{\omega}_{h} \times \bar{\omega}_{\tau}$,
where
$\bar{\omega}_{h} = \{ x_i =ih, ~i = 0, 1, \cdots, N; ~x_0 = 0, ~x_N = L \}$ and
$\bar{\omega}_{\tau} = \{ t_j = j \tau, ~j = 0, 1, \cdots, M; ~t_M = T \}$,
in which $h = \frac{L}{N}$, $\tau = \frac{T}{M}$ are the uniform spatial and temporal mesh sizes respectively,
and $N$, $M$ are two positive integers. Let
\begin{equation*}
\mathcal{S}_{h} = \{ \bm{v} \mid \bm{v} = (v_0, v_1, \cdots, v_N), ~v_0 = v_N = 0 \}
\end{equation*}
be defined on $\bar{\omega}_{h}$.
Then about the discretization of Caputo fractional derivative,
we utilize the $L2$-$1 \sigma$ formula derived by Alikhanov in \cite{Alikhanov2015A},
and some helpful properties for later analysis in next subsection are reviewed therewith.
\begin{lemma}(\cite[Lemma 2]{Alikhanov2015A})
Suppose $0< \alpha <1$, $\sigma = 1 - \frac{\alpha}{2}$, $y(t) \in C^3 [0,T]$, and
$t_{j + \sigma} = (j + \sigma) \tau$.
Then
\begin{equation*}
\left |  D_{0,t_{j + \sigma}}^{\alpha} y(t)
- \Delta^{\alpha}_{0,t_{j + \sigma}} y(t) \right | = \mathcal{O}(\tau^{3 - \alpha}),
\end{equation*}
where
\begin{equation}
\Delta^{\alpha}_{0,t_{j + \sigma}} y(t) =
\frac{\tau^{-\alpha}}{\Gamma(2 - \alpha)} \sum\limits^{j}_{s = 0} c^{(\alpha,\sigma)}_{j - s}
\left[ y(t_{s + 1}) - y(t_{s}) \right],
\label{eq2.1}
\end{equation}
and for $j = 0$,
\begin{equation*}
c^{(\alpha,\sigma)}_{0} = a^{(\alpha,\sigma)}_{0},
\end{equation*}
for $j \geq 1$,
\begin{equation*}
c^{(\alpha,\sigma)}_{s}=
\begin{cases}
a^{(\alpha,\sigma)}_{0} + b^{(\alpha,\sigma)}_{1}, & s = 0,\\
a^{(\alpha,\sigma)}_{s} + b^{(\alpha,\sigma)}_{s + 1} - b^{(\alpha,\sigma)}_{s}, & 1\leq s \leq j - 1,\\
a^{(\alpha,\sigma)}_{j} - b^{(\alpha,\sigma)}_{j}, & s = j.
\end{cases}
\end{equation*}
In which
\begin{equation*}
a^{(\alpha,\sigma)}_{0} = \sigma^{1 - \alpha}, ~ a^{(\alpha,\sigma)}_{l}
= (l + \sigma)^{1 - \alpha} - (l - 1 + \sigma)^{1 - \alpha} ~( l \geq 1),
\end{equation*}
and
\begin{equation*}
b^{(\alpha,\sigma)}_{l} = \frac{1}{2 - \alpha}
\left[(l + \sigma)^{2 - \alpha} - (l - 1 + \sigma)^{2 - \alpha} \right]
- \frac{1}{2} \left[ (l + \sigma)^{1 - \alpha} - (l - 1 + \sigma)^{1 - \alpha} \right].
\end{equation*}
\label{lemma2.1}
\end{lemma}

Here, the properties of $c^{(\alpha,\sigma)}_{j}$
proved in \cite{Alikhanov2015A} are revisited as below.
\begin{lemma}(\cite[Lemma 4]{Alikhanov2015A})
For any $\alpha ~ (0 < \alpha < 1)$ and $c^{(\alpha,\sigma)}_{j}$ defined in Lemma \ref{lemma2.1},
it holds
\begin{equation*}
c^{(\alpha,\sigma)}_{j} > \frac{1 - \alpha}{2} (j + \sigma)^{- \alpha},
\end{equation*}
\begin{equation*}
c^{(\alpha,\sigma)}_{0} > c^{(\alpha,\sigma)}_{1} > c^{(\alpha,\sigma)}_{2}
> \cdots > c^{(\alpha,\sigma)}_{j - 1} > c^{(\alpha,\sigma)}_{j},
\end{equation*}
where $\sigma = 1 - \frac{\alpha}{2}$.
\label{lemma2.2}
\end{lemma}

As for approximation of the time drift term $\frac{\partial u(x,t)}{\partial t}$,
the Taylor expansion of the function $u(t)$ is employed
for $t = t_{j + 1}, t_j$ and $t_{j - 1}$ at the point $t_{j + \sigma}$, respectively.
Thus, the next lemma can be easily obtained after simple calculation.
\begin{lemma}
Suppose $u \in C^3 [0,T]$, we have
\begin{equation*}
\begin{split}
\delta_{\hat{t}} u(t_j) & \equiv \frac{1}{2 \tau} [(2 \sigma + 1) u(t_{j + 1}) - 4 \sigma u(t_j) + (2 \sigma - 1) u(t_{j - 1})] \\
& = \frac{d u(t_{j + \sigma})}{d t} + \mathcal{O}(\tau^2), ~ j \geq 1.
\end{split}
\end{equation*}
\label{lemma2.3}
\end{lemma}

On the other hand, \cite{Alikhanov2015A} also proved that
\begin{equation}
\mathcal{L} u(x,t) \mid_{(x_i, t_{j + \sigma})} = \sigma \Lambda u(x_i, t_{j + 1}) + (1 - \sigma) \Lambda u(x_i, t_j)
 + \mathcal{O}(\tau^2 + h^2),
 \label{eq2.2}
\end{equation}
where $\Lambda$ is a difference operator, which approximates the continuous operator $\mathcal{L}$, defined by
\begin{align*}
\Lambda u(x_i, t_j) = \frac{1}{h^2} \Big[ & k(x_{i + \frac{1}{2}}, t_{j + \sigma}) u(x_{i + 1}, t_j)
- \left( k(x_{i + \frac{1}{2}}, t_{j + \sigma}) + k(x_{i - \frac{1}{2}}, t_{j + \sigma}) \right) u(x_i, t_j) \\
& + k(x_{i - \frac{1}{2}}, t_{j + \sigma}) u(x_{i - 1}, t_j) \Big]  - q(x_i, t_{j + \sigma}) u(x_i,t_j).
\end{align*}

Assume $u(x,t)$ is a sufficiently smooth solution of the TFRDE\/ \eqref{eq1.1}.
For the sake of simplification, some symbols are introduced:
\begin{equation*}
u_{i}^{j + \sigma} = \sigma u_{i}^{j+1} + (1 - \sigma) u_{i}^{j}, \quad
k_{i}^{j + \sigma} = k(x_i, t_{j + \sigma}), \quad
q_{i}^{j + \sigma} = q(x_i, t_{j + \sigma}), \quad
f_{i}^{j + \sigma} = f(x_i,t_{j + \sigma}).
\end{equation*}

Using\/ \eqref{eq2.1}-\eqref{eq2.2} and omitting the small term,
the solution of\/ \eqref{eq1.1} can be approximated by the following IDS
for $(x,t) = (x_i,t_{j + \sigma}) \in \bar{\omega}_{h \tau}$, $i = 1, 2, \ldots, N-1$,
$j = 0, 1, \ldots, M-1$:
\begin{equation*}
\delta_{\hat{t}} u_{i}^{j} + \Delta^{\alpha}_{0,t_{j + \sigma}} u_i = \Lambda u_{i}^{j + \sigma}
+ f_{i}^{j + \sigma}.
\end{equation*}

There is a problem that cannot be ignored in the above equation:
when $ j = 0$, then $u_i^{j - 1} = u_i^{-1}$ is defined outside of $[0, T]$.
In numerical calculation, we handle with this problem mainly
by using the neighbouring function values to approximate $u_i^{-1}$,
that is,
\begin{equation*}
u_i^{-1} = u_i^0 - \tau \frac{\partial u_i^0}{\partial t} + \mathcal{O}(\tau^2).
\end{equation*}
If $\frac{\partial u_i^0}{\partial t} \neq 0$, our IDS only has first-order temporal accuracy.
Thus, in order to obtain the second-order accuracy in time,
we suppose $\frac{\partial u(x,0)}{\partial t} = 0$, then set $u_i^{-1} = u_i^0$.
The IDS with the accuracy order $\mathcal{O}(\tau^2 + h^2)$ is:
\begin{equation}
\begin{cases}
\delta_{\hat{t}} u_{i}^{j} + \Delta^{\alpha}_{0,t_{j + \sigma}} u_i = \Lambda u_{i}^{j + \sigma}
 + f_{i}^{j + \sigma},  & 1 \leq i \leq N-1,~~ 0 \leq j \leq M-1, \\
u_{i}^{0} = u_{0}(x_i), ~ u_i^{-1} = u_i^0, & 0 \leq i \leq N, \\
u_{0}^{j} = \phi_1(t_j), ~ u_{N}^{j} = \phi_2(t_j),  & 0 \leq j \leq M.
\end{cases}
\label{eq2.3}
\end{equation}

It is interesting to note that for $\alpha \rightarrow 1$,
the Crank-Nicolson difference scheme is obtained.
In the next subsection, we will proof the unconditional stability and give error estimate about this approximate scheme.
\subsection{Stability analysis and optimal error estimates}
\label{sec2.2}
Before exploring the stability and convergence of Eq. \eqref{eq2.3},
an inner product and the corresponding norm are introduced:
\begin{equation*}
(\bm{u}, \bm{v}) = h \sum_{i = 1}^{N - 1} u_i v_i, ~~ \| \bm{u} \| = \sqrt{(\bm{u}, \bm{u})},
\end{equation*}
here $\bm{u}, \bm{v} \in \mathcal{S}_h$ are arbitrary vectors.
Meanwhile, we need another two lemmas, which are essential for our proof,
see \cite{Alikhanov2015A, Sun2015Some}.
\begin{lemma}(\cite[Corollary 1]{Alikhanov2015A})
Let $V_\tau = \{ \bm{u} \mid \bm{u} = (u^0, u^1, \ldots, u^M) \}$.
For any $\bm{u} \in V_\tau$, one has the following inequality
\begin{equation*}
\Big[ \sigma u^{j + 1} + (1 - \sigma) u^j \Big] \Delta_{0,t_{j + \sigma}}^{\alpha} \bm{u}
\geq \frac{1}{2} \Delta_{0,t_{j + \sigma}}^{\alpha} (\bm{u})^2.
\end{equation*}
\label{lemma2.4}
\end{lemma}
\begin{lemma}(\cite[Lemma 3.5]{Sun2015Some})
For any grid functions $u^0, u^1, \cdots, u^N \in \mathcal{S}_h$, we have
\begin{equation*}
(\delta_{\hat{t}} u^k, \sigma u^{k + 1} + (1 - \sigma) u^k) \geq \frac{1}{4 \tau} (E^{k + 1} - E^k), ~k \geq 1,
\end{equation*}
with
\begin{equation*}
E^{k + 1} = (2 \sigma + 1) \| u^{k + 1} \|^2 - ( 2 \sigma - 1) \| u^{k} \|^2
+ (2 \sigma^2 + \sigma - 1) \| u^{k + 1} - u^k \|^2, ~ k \geq 0.
\end{equation*}
In addition, it holds
\begin{equation*}
E^{k + 1} \geq \frac{1}{\sigma} \| u^{k + 1} \|^2, ~ k \geq 0.
\end{equation*}
\label{lemma2.5}
\end{lemma}

From Lemma \ref{lemma2.4}, we obtain $E^0 = 2 \| u^{0} \|^2$.
With this in hand, the next theorem can be established.
\begin{theorem}
Denote $u^{j + 1} = (u_1^{j + 1}, u_2^{j + 1}, \ldots, u_{N - 1}^{j + 1})^T$ and
$\left \| f^{j + \sigma} \right \|^2 = h \sum\limits_{i = 1}^{N - 1} f^2 (x_i, t_{j + \sigma})$.
Then the IDS\/ \eqref{eq2.3}
is unconditionally stable, and the following two priori estimates hold:
\begin{equation}
\| u^{1} \|^2
\leq
\Big( \frac{4 \sigma^{\alpha} T^{1 - \alpha}}{\Gamma(2 - \alpha)} + 2 \sigma \Big) \| u^0 \|^2
+ 4 \sigma^{\alpha} T^{1 + \alpha} \Gamma(2 - \alpha) \| f^{\sigma} \|^2,
\label{eq2.4}
\end{equation}
\begin{equation}
 \| u^{k} \|^2
\leq
C_1 \| u^{1} \|^2  + \Big( \frac{2 \sigma T^{1 - \alpha}}{\Gamma(2 - \alpha)} + 4 \sigma^3 \Big) \| u^{0} \|^2
 + 8 \sigma T^{1 + \alpha} \Gamma(1 - \alpha) \sum_{j = 1}^{k - 1} \| f^{j + \sigma} \|^2, ~ k \geq 2,
\label{eq2.5}
\end{equation}
where $C_1 = \frac{T^{1 - \alpha} \sigma}{\Gamma(1 - \alpha)}
+ \frac{2 \sigma (4 - 3 \alpha) T^{1 - \alpha}}{\Gamma(3 - \alpha)} + 4 \sigma^3 + 4 \sigma^2 - \sigma$.
\label{th2.1}
\end{theorem}
\textbf{Proof.}
Taking the inner product of \/ \eqref{eq2.3} with $u^{j + \sigma}$, it has
\begin{equation*}
(\delta_{\hat{t}} u^j, u^{j + \sigma}) + (\Delta_{0,t_{j + \sigma}}^{\alpha} u, u^{j + \sigma}) =
(\Lambda u^{j + \sigma}, u^{j + \sigma}) + (f^{j + \sigma}, u^{j + \sigma}).
\end{equation*}

Using Lemmas \ref{lemma2.4}-\ref{lemma2.5} and noticing $(\Lambda u^{j + \sigma}, u^{j + \sigma}) \leq 0$,
it can be obtained that
\begin{equation}
\frac{1}{4 \tau} (E^{j + 1} - E^j) + \frac{1}{2} \Delta_{0,t_{j + \sigma}}^{\alpha} \left \| u \right \|^2
\leq (f^{j + \sigma}, u^{j + \sigma}).
\label{eq2.6}
\end{equation}
\textbf{Step 1.} When $j = 0$, from the inequality (\ref{eq2.6}), we have
\begin{equation*}
\frac{1}{4 \tau} (E^{1} - E^0)
+ \frac{1}{2\tau^{\alpha} \Gamma(2 - \alpha)} a_0^{(\alpha, \sigma)} ( \| u^1 \|^2 - \| u^0 \|^2)
\leq (f^{\sigma}, u^{\sigma}).
\end{equation*}
With the help of virtue Cauchy-Schwarz inequality, we arrive at
\begin{align*}
\| u^1 \|^2 + \frac{2 \tau \sigma}{T^{\alpha} \Gamma(2 - \alpha)} a_0^{(\alpha, \sigma)} \| u^1 \|^2
\leq & \frac{2 \tau \sigma}{T^{\alpha} \Gamma(2 - \alpha)} a_0^{(\alpha, \sigma)} \| u^0 \|^2
+ 2 \sigma \| u^0 \|^2 + \frac{\tau \sigma}{\varepsilon_1} \| f^{\sigma} \|^2 \\
&\quad + 8 \tau \sigma \varepsilon_1 (\| u^1 \|^2 + \| u^0 \|^2) , ~ \varepsilon_1 > 0.
\end{align*}
Let $\varepsilon_1 = \frac{1}{4 T^{\alpha} \Gamma(2 - \alpha)} a_0^{(\alpha, \sigma)}$,
it gives immediately the estimate for $u^1$, that is
\begin{align*}
\| u^1 \|^2
& \leq
\Big( \frac{4 \tau \sigma}{T^{\alpha} \Gamma(2 - \alpha) a_0^{(\alpha, \sigma)}} + 2 \sigma \Big) \| u^0 \|^2
+ \frac{4 \tau \sigma T^{\alpha} \Gamma(2 - \alpha)}{a_0^{(\alpha, \sigma)}} \| f^{\sigma} \|^2 \\
& \leq
\Big( \frac{4 T^{1 - \alpha} \sigma^{\alpha}}{\Gamma(2 - \alpha)} + 2 \sigma \Big) \| u^0 \|^2
+ 4 \sigma^{\alpha} T^{1 + \alpha} \Gamma(2 - \alpha) \| f^{\sigma} \|^2.
\end{align*}
\textbf{Step 2.} When $j \geq 1$, summing up for $j$ in (\ref{eq2.6}) from $1$ to $k - 1$ and doing some simple manipulations,
it obtains
\begin{align}
\frac{1}{4 \tau} (E^{k} & - E^1)
+ \frac{1}{2 \tau^{\alpha} \Gamma(2 - \alpha)} \Big[ c_0^{(\alpha, \sigma)} \sum_{j = 1}^{k - 1} \| u^{j + 1} \|^2
- \sum_{j = 1}^{k - 1} \sum_{s = 2}^{j} ( c_{j - s}^{(\alpha, \sigma)} - c_{j - s + 1}^{(\alpha, \sigma)} ) \| u^{s} \|^2 \Big]
\nonumber \\
& \leq \frac{1}{2 \tau^{\alpha} \Gamma(2 - \alpha)} \| u^{0} \|^2 \sum_{j = 1}^{k - 1} c_{j}^{(\alpha, \sigma)}
+ \frac{1}{2 \tau^{\alpha} \Gamma(2 - \alpha)} \| u^{1} \|^2 \sum_{j = 1}^{k - 1} ( c_{j - 1}^{(\alpha, \sigma)} - c_{j}^{(\alpha, \sigma)} )
\nonumber \\
&\qquad + \sum_{j = 1}^{k - 1} \| f^{j + \sigma} \| \cdot ( \sigma \| u^{j + 1} \| + (1 - \sigma) \| u^{j} \|).
\label{eq2.7}
\end{align}
To estimate the second term on the left hand side of inequality\/ \eqref{eq2.7},
Lemma \ref{lemma2.2} is applied.
Then
\begin{align*}
&\frac{1}{2 \tau^{\alpha} \Gamma(2 - \alpha)} \Big[ c_0^{(\alpha, \sigma)} \sum_{j = 1}^{k - 1} \| u^{j + 1} \|^2
- \sum_{j = 1}^{k - 1} \sum_{s = 2}^{j} ( c_{j - s}^{(\alpha, \sigma)} - c_{j - s + 1}^{(\alpha, \sigma)} ) \| u^{s} \|^2 \Big]
\\ &\quad = \frac{1}{2 \tau^{\alpha} \Gamma(2 - \alpha)} \sum_{j = 2}^{k} c_{k - j}^{(\alpha, \sigma)} \| u^{j} \|^2
\geq
\frac{1}{2 \tau^{\alpha} \Gamma(2 - \alpha)} \frac{1 - \alpha}{2} (j - 1 + \sigma)^{- \alpha} \sum_{j = 2}^{k} \| u^{j} \|^2\\
&\quad \geq
\frac{1}{4 T^{\alpha} \Gamma(1 - \alpha)} \sum_{j = 2}^{k} \| u^{j} \|^2.
\end{align*}
Bringing above estimate to inequality\/ \eqref{eq2.7} gives
\begin{align}
& \| u^{k} \|^2 + \frac{\tau \sigma}{T^{\alpha} \Gamma(1 - \alpha)} \sum_{j = 2}^{k} \| u^{j} \|^2
\nonumber \\
&~~ \leq \sigma E^1
+ \frac{2 \tau^{1 - \alpha} \sigma (4 - 3 \alpha) (k - 1 + \sigma)^{1 - \alpha}}{\Gamma(3 - \alpha)} \| u^{1} \|^2
+ \frac{2 \tau^{1 - \alpha} \sigma (k - 1 + \sigma)^{1 - \alpha}}{\Gamma(2 - \alpha)} \| u^{0} \|^2
\nonumber \\
&\qquad\qquad + 4 \tau \sigma \varepsilon_2 \sum_{j = 1}^{k - 1} \Big( \sigma \| u^{j + 1} \| + (1 - \sigma) \| u^{j} \| \Big)^2
+ \frac{\tau \sigma}{\varepsilon_2} \sum_{j = 1}^{k - 1} \| f^{j + \sigma} \|^2
\nonumber \\
&~~ \leq \sigma \Big[ (4 \sigma^2 + 4 \sigma - 1) \| u^{1} \|^2 + 4 \sigma^2 \| u^{0} \|^2 \Big]
+ \frac{2 \sigma (4 - 3 \alpha) T^{1 - \alpha}}{\Gamma(3 - \alpha)} \| u^{1} \|^2
\nonumber \\
&\qquad\qquad + \frac{2 \sigma T^{1 - \alpha}}{\Gamma(2 - \alpha)} \| u^{0} \|^2
+ 8 \tau \sigma \varepsilon_2 \sum_{j = 1}^{k} \| u^{j} \|^2
+ \frac{\tau \sigma}{\varepsilon_2} \sum_{j = 1}^{k - 1} \| f^{j + \sigma} \|^2, ~ \varepsilon_2 > 0,
\label{eq2.8}
\end{align}
where $E^1 \leq (4 \sigma^2 + 4 \sigma - 1) \| u^1 \|^2 + 4 \sigma^2 \| u^0 \|^2$.
Taking $\varepsilon_2 = \frac{1}{8 T^{\alpha} \Gamma(1 - \alpha)}$,  inequality\/ \eqref{eq2.8} leads to
\begin{align*}
\| u^{k} \|^2
& \leq
\Big[ \frac{\tau \sigma}{T^{\alpha} \Gamma(1 - \alpha)} + \frac{2 \sigma  (4 - 3 \alpha) T^{1 - \alpha}}{\Gamma(3 - \alpha)}
 + 4 \sigma^3 + 4 \sigma^2 - \sigma \Big] \| u^{1} \|^2 \\
&\qquad + \Big( \frac{2 \sigma T^{1 - \alpha}}{\Gamma(2 - \alpha)}
+ 4 \sigma^3 \Big) \| u^{0} \|^2 + 8 \tau \sigma T^{\alpha} \Gamma(1 - \alpha) \sum_{j = 1}^{k - 1} \| f^{j + \sigma} \|^2 \\
& \leq
\Big[ \frac{T^{1 - \alpha} \sigma}{\Gamma(1 - \alpha)} + \frac{2 \sigma (4 - 3 \alpha) T^{1 - \alpha}}{\Gamma(3 - \alpha)}
 + 4 \sigma^3 + 4 \sigma^2 - \sigma \Big] \| u^{1} \|^2 \\
&\qquad + \Big( \frac{2 \sigma T^{1 - \alpha}}{\Gamma(2 - \alpha)}
+ 4 \sigma^3 \Big) \| u^{0} \|^2 + 8 \sigma T^{1 + \alpha} \Gamma(1 - \alpha) \sum_{j = 1}^{k - 1} \| f^{j + \sigma} \|^2.
\end{align*}
The proof of Theorem \ref{th2.1} is completed. \hfill $\Box$

With the above proof, the convergence of the difference scheme\/ \eqref{eq2.3} is easy to obtain.
\begin{theorem}
Let $u(x,t)$ be the sufficiently smooth exact solution of\/ \eqref{eq1.1},
$\{ u_i^j \mid x_i \in \bar{\omega}_h, 0 \leq j \leq M \}$ be the solution
of the problem\/ \eqref{eq2.3}. Let
$e_i^j = u(x_i, t_j) - u_i^j  ~(0 \leq i \leq N, ~ 0 \leq j \leq M)$
and $e^j = [e_1^j, e_2^j, \cdots, e_{N - 1}^j]^{T}  ~(0 \leq j \leq M)$.
Then, for $j = 0, 1, 2, \cdots, M$, we have
\begin{equation*}
\| e^j \| \leq C_2 (\tau^2 + h^2), ~ 0 \leq j \leq M,
\end{equation*}
where $C_2$ is a positive constant, which may depend on $\alpha$ and $T$.
\label{th2.2}
\end{theorem}
\textbf{Proof.}
Subtracting\/ \eqref{eq2.3} from\/ \eqref{eq1.1}, the error equations are represented as:
\begin{equation*}
\begin{cases}
\delta_{\hat{t}} e_{i}^{j} + \Delta^{\alpha}_{0,t_{j + \sigma}} e_i = \Lambda e_{i}^{j + \sigma}
 + R_{i}^{j + \sigma},  & 1 \leq i \leq N-1,~~ 0 \leq j \leq M-1, \\
e_{i}^{-1} = e_{i}^{0} = 0,  & 0 \leq i \leq N, \\
e_{0}^{j} = e_{N}^{j} = 0,  & 0 \leq j \leq M,
\end{cases}
\end{equation*}
with $R_i^j = \mathcal{O}(\tau^2 + h^2)$. Then the following procedure is similar to Theorem \ref{th2.1},
the error $e^j$ yields
\begin{equation*}
\| e^j \| \leq C_2 (\tau^2 + h^2), ~ 0 \leq j \leq M,
\end{equation*}
where $C_2$ is a positive constant, which may depend on $\alpha$ and $T$.
\hfill $\Box$

Theorem \ref{th2.2} implies that our numerical scheme converges to
the optimal order $\mathcal{O}(\tau^2 + h^2)$ in the $L_2$-norm,
when the solution of Eq. \eqref{eq1.1} is sufficiently smooth.
If the solution of Eq. \eqref{eq1.1} is non-smooth,
several interesting alternative approaches \cite{zeng2017nonsmooth,liao2018nonsmooth}
have been introduced to address this problem.

For convenience, Eq. \eqref{eq2.3} can be rewritten into the equivalent matrix form:
\begin{equation}
\begin{cases}
\begin{split}
\mathcal{M}^{j + 1} u^{j + 1} = & B^{j} u^{j} - h^2 (2 \sigma - 1) u^{j - 1}
- \frac{2 \tau^{1 - \alpha} h^2}{\Gamma(2 - \alpha)}
\sum_{s = 0}^{j - 1} c_{j - s}^{(\alpha, \sigma)} (u^{s + 1} - u^{s}) \\
&\quad + 2 \tau h^2 f^{j + \sigma} + \eta^{j + \sigma}, \quad 0 \leq j \leq M - 1,
\end{split} \\
u^{0} = u_0,
\end{cases}
\label{eq2.9}
\end{equation}
where $f^{j + \sigma} = [f_{1}^{j + \sigma}, f_{2}^{j + \sigma}, \cdots, f_{N - 1}^{j + \sigma}]^{T}$,
$u_0 = [u_{0} (x_1), u_{0} (x_2), \cdots, u_{0} (x_{N - 1}) ]^{T}$,
\begin{equation*}
\eta^{j + \sigma} = 2 \tau [\sigma k_{1 - \frac{1}{2}}^{j + \sigma} u_0^{j + 1}
+ (1 - \sigma) k_{1 - \frac{1}{2}}^{j + \sigma} u_0^{j}, 0, \cdots, 0, \sigma k_{N - 1 + \frac{1}{2}}^{j + \sigma} u_N^{j + 1}
+ (1 - \sigma) k_{N - 1 + \frac{1}{2}}^{j + \sigma} u_N^{j} ]^{T},
\end{equation*}
and
\begin{equation*}
\mathcal{M}^{j + 1} = \Big[ h^2 (2 \sigma + 1)
+ \frac{2 \tau^{1 - \alpha} h^2}{\Gamma(2 - \alpha)} c_{0}^{(\alpha, \sigma)} \Big] I
- 2 \sigma \tau (A^{j + \sigma} - h^2 Q^{j + \sigma}),
\end{equation*}
\begin{equation*}
B^j = \Big( 4 h^2 \sigma + \frac{2 \tau^{1 - \alpha} h^2}{\Gamma(2 - \alpha)} c_{0}^{(\alpha, \sigma)} \Big) I
+ 2 (1 - \sigma) \tau (A^{j + \sigma} - h^2 Q^{j + \sigma}).
\end{equation*}
Whereas
\begin{align*}
A^{j + \sigma} = & - \textrm{diag} \Big( \Big[ (k_{1 - \frac{1}{2}}^{j + \sigma} + k_{1 + \frac{1}{2}}^{j + \sigma}),
(k_{2 - \frac{1}{2}}^{j + \sigma} + k_{2 + \frac{1}{2}}^{j + \sigma}), \cdots,
(k_{N - 1 - \frac{1}{2}}^{j + \sigma} + k_{N - 1 + \frac{1}{2}}^{j + \sigma}) \Big] \Big) \\
& + \textrm{diag} \Big( \Big[ k_{2 - \frac{1}{2}}^{j + \sigma}, k_{3 - \frac{1}{2}}^{j + \sigma}, \cdots,
k_{N - 1 - \frac{1}{2}}^{j + \sigma} \Big], -1 \Big)
+ \textrm{diag} \Big( \Big[ k_{1 + \frac{1}{2}}^{j + \sigma}, k_{2 + \frac{1}{2}}^{j + \sigma}, \cdots,
k_{N - 2 + \frac{1}{2}}^{j + \sigma}\Big], 1 \Big),
\end{align*}
$Q^{j + \sigma} = \textrm{diag}(q_{1}^{j + \sigma}, q_{2}^{j + \sigma}, \cdots, q_{N - 1}^{j + \sigma})$
and $I$ is the identity matrix with an appropriate size.
Upon above definitions, it is obvious that the coefficient matrix $\mathcal{M}^{j + 1}$ is a symmetric tridiagonal matrix.

\section{The two-dimensional problem of TFRDE}
\label{sec3}
In practical applications, one-dimensional problems are rare, therefore in this section,
the two-dimensional (2D) TFRDE is studied:
\begin{equation}
\frac{\partial u(x,y,t)}{\partial t} + D_{0,t}^{\alpha} u(x,y,t)
=  \mathcal{N} u(x,y,t) + f(x,y,t),  ~ (x,y) \in  [0, L_x] \times [0, L_y], ~  0 \leq t \leq T,
\label{eq3.1}
\end{equation}
with initial condition
\begin{equation}
u(x,y,0) = u_{0}(x,y), ~ (x,y) \in [0, L_x] \times [0, L_y],
\label{eq3.2}
\end{equation}
and boundary value conditions
\begin{equation}
u(0,y,t) = \psi_1(y,t), ~ u(L_x,y,t) = \psi_2(y,t),  ~ 0 \leq t \leq T,
\label{eq3.3}
\end{equation}
\begin{equation}
u(x,0,t) = g_1(x,t), ~ u(x,L_y,t) = g_2(x,t), ~ 0 \leq t \leq T,
\label{eq3.4}
\end{equation}
where
\begin{equation*}
 \mathcal{N} u(x,y,t)  = \frac{\partial}{\partial x} \Big( d(x,y,t) \frac{\partial u(x,y,t)}{\partial x} \Big)
+ \frac{\partial}{\partial y} \Big( k(x,y,t) \frac{\partial u(x,y,t)}{\partial y} \Big) - q(x,y,t) u(x,y,t),
\end{equation*}
$d(x,y,t) \geq C_3 > 0$, $k(x,y,t) \geq C_4 > 0$, $q(x,y,t) \geq 0$ and $f(x,y,t)$ are sufficiently smooth functions.
In the rest of this section, we will deduce a second-order difference scheme and investigate its stability and convergence.

\subsection{Difference scheme for the 2D TFRDE}
\label{sec3.1}
Taking two positive integers $N_x$ and $N_y$, then $h_x = \frac{L_x}{N_x}$, $h_y = \frac{L_y}{N_y}$.
Denote
\begin{equation*}
\hat{\omega} = \{ x_i = i h_x, ~y_l = l h_y, ~0 \leq i \leq N_x, ~0 \leq l \leq N_y;
~x_0 = x_{N_x} = 0, ~y_0 = y_{N_y} = 0 \},
\end{equation*}
and
\begin{equation*}
\hat{\mathcal{S}} = \{ \bm{v} \mid \bm{v} = (v_{il})_{0 \leq i \leq N_x, ~0 \leq l \leq N_y};
~v_{0 l} = v_{N_x l} = 0, ~v_{i 0} = v_{i N_y} = 0\}.
\end{equation*}

Now the fully discrete scheme is derived.
Let
\begin{align*}
\tilde{\Lambda} u(x_i, y_l ,t_j)
& = \frac{1}{h_x^2} \Big[ d(x_{i - \frac{1}{2}},y_{l}, t_{j + \sigma}) u(x_{i - 1}, y_{l},t_{j})
- \Big( d(x_{i - \frac{1}{2}},y_{l}, t_{j + \sigma})
+ d(x_{i + \frac{1}{2}},y_{l}, t_{j + \sigma}) \Big) \\
&\quad \times u(x_{i}, y_{l},t_{j}) + d(x_{i + \frac{1}{2}},y_{l}, t_{j + \sigma})
u(x_{i + 1}, y_{l},t_{j}) \Big]
+ \frac{1}{h_y^2} \Big[ k(x_{i},y_{l - \frac{1}{2}}, t_{j + \sigma})\\
&\quad \times u(x_{i}, y_{l - 1},t_{j}) - \Big( k(x_{i},y_{l - \frac{1}{2}}, t_{j + \sigma})
+ k(x_{i},y_{l + \frac{1}{2}}, t_{j + \sigma}) \Big) u(x_{i}, y_{l},t_{j})\\
&\quad + k(x_{i},y_{l + \frac{1}{2}}, t_{j + \sigma}) u(x_{i}, y_{l + 1},t_{j}) \Big]
- q(x_i, y_l, t_{j + \sigma}) u(x_i, y_l, t_j)
\end{align*}
be a difference operator approximates the continuous operator $\mathcal{N}$.
Afterwards, similar implementation as presented in Eq. \eqref{eq2.2}, we have
\begin{equation*}
\mathcal{N} u(x, y, t) \mid_{(x_i, y_l , t_{ j + \sigma})}
 = \sigma \tilde{\Lambda} u(x_i, y_l, t_{j + 1}) + (1 - \sigma)  \tilde{\Lambda} u(x_i, y_l, t_{j})
 + \mathcal{O}(\tau^2 + h_{x}^{2} + h_{y}^{2}),
\end{equation*}
and some other new notations are given based on Section 2
\begin{equation*}
u_{i l}^{j + \sigma} = \sigma u_{i l}^{j + 1} + (1 - \sigma) u_{i l}^{j},  \quad
d_{i l}^{j + \sigma} = d(x_i, y_l , t_{j + \sigma}), \quad
k_{i l}^{j + \sigma} = k(x_i, y_l , t_{j + \sigma}),
\end{equation*}
\begin{equation*}
q_{i l}^{j + \sigma} = q(x_i, y_l , t_{j + \sigma}), \quad
f_{i l}^{j + \sigma} = f(x_i, y_l , t_{j + \sigma}).
\end{equation*}

Similar to the process of dealing with 1D case in Section 2 obtains
\begin{equation*}
\delta_{\hat{t}} u_{i l}^{j} + \Delta_{0, t_{j + \sigma}}^{\alpha}
= \tilde{\Lambda} u_{i l}^{j + \sigma} + f_{i l}^{j + \sigma}.
\end{equation*}
When $j = 0$, $u_{i l}^{j - 1} = u_{i l}^{-1}$ is defined outside of $[0,T]$,
in the same way as Section 2,
\begin{equation*}
u_{i l}^{-1} = u_{i l}^0 - \tau \frac{\partial u_{i l}^0}{\partial t} + \mathcal{O}(\tau^2).
\end{equation*}
In order to obtain the second-order accuracy in time, we assume
$\frac{\partial u(x, y, 0)}{\partial t} = 0$,
thus set $u_{i l}^{-1} = u_{i l}^{0}$.
Adding the discrete initial-boundary conditions,
our approximate scheme for the problem\/ \eqref{eq3.1}-\eqref{eq3.4} is
\begin{equation}
\begin{cases}
\delta_{\hat{t}} u_{i l}^{j} + \Delta_{0, t_{j + \sigma}}^{\alpha} = \tilde{\Lambda} u_{i l}^{j + \sigma}
+ f_{i l}^{j + \sigma}, & 1 \leq i \leq N_x - 1, 1 \leq l \leq N_y - 1, 0 \leq j \leq M, \\
u_{i l}^{0} = u_0(x_i, y_l), ~ u_{i l}^{-1} = u_{i l}^{0}, & 0 \leq i \leq N_x, 0 \leq l \leq N_y, \\
u_{0 l}^{j} = \psi_1(y_l, t_j), ~ u_{N_x l}^{j} = \psi_2(y_l, t_j),  & 0 \leq l \leq N_y, 0 \leq j \leq M, \\
u_{i 0}^{j} = g_1(x_i, t_j), ~ u_{i N_y}^{j} = g_2(x_i, t_j),  & 0 \leq i \leq N_x, 0 \leq j \leq M.
\end{cases}
\label{eq3.5}
\end{equation}

\subsection{Stability and convergence analysis}
\label{sec3.2}
In order to probe into the scheme\/ \eqref{eq3.5},
an inner product and the corresponding norm are defined to facilitate our subsequent analysis

\begin{equation*}
(\bm{u}, \bm{v}) = h_x h_y \sum_{i = 1}^{N_x - 1} \sum_{l = 1}^{N_y - 1}u_{i l} v_{i l},
~~ \| \bm{u} \| = \sqrt{(\bm{u}, \bm{u})}, ~~ \forall \bm{u}, \bm{v} \in \hat{\mathcal{S}}.
\end{equation*}
The priori estimate of\/ \eqref{eq3.5} is given.
\begin{theorem}
Suppose $\{ U_{i l}^{j + 1} \mid  0 \leq i \leq N_x, ~0 \leq l \leq N_y, ~0 \leq j \leq M \}$
be the solution of \/ \eqref{eq3.5} and denote
$\left \| f^{j + \sigma} \right \|^2
= h_x h_y \sum\limits_{i = 1}^{N_x - 1}\sum\limits_{l = 1}^{N_y - 1} f^2 (x_i, y_l, t_{j + \sigma})$.
Then the IDS\/ \eqref{eq3.5}
is unconditionally stable, and the following two priori estimates hold:
\begin{equation}
\| U^{1} \|^2
\leq
\Big( \frac{4 T^{1 - \alpha} \sigma^{\alpha}}{\Gamma(2 - \alpha)} + 2 \sigma \Big) \| U^0 \|^2
+ 4 \sigma^{\alpha} T^{1 + \alpha} \Gamma(2 - \alpha) \| f^{\sigma} \|^2,
\label{eq3.6}
\end{equation}
\begin{equation}
 \| U^{k} \|^2
\leq
C_1 \| U^{1} \|^2  + \Big( \frac{2 \sigma T^{1 - \alpha}}{\Gamma(2 - \alpha)} + 4 \sigma^3 \Big) \| U^{0} \|^2
 + 8 \sigma T^{1 + \alpha} \Gamma(1 - \alpha) \sum_{j = 1}^{k - 1} \| f^{j + \sigma} \|^2, ~ k \geq 2,
\label{eq3.7}
\end{equation}
where $C_1$ is given in Theorem \ref{th2.1}.
\label{th3.1}
\end{theorem}
\textbf{Proof.}
In this proof, we take advantage of the method in Theorem \ref{th2.1} again.
Taking the inner product of\/ \eqref{eq3.5} with $U^{j + \sigma} = \sigma U^{j + 1} + (1 - \sigma) U^{j}$,
it results
\begin{equation*}
(\delta_{\hat{t}} U^j, U^{j + \sigma}) + (\Delta_{0,t_{j + \sigma}}^{\alpha} U, U^{j + \sigma}) =
(\tilde{\Lambda} U^{j + \sigma}, U^{j + \sigma}) + (f^{j + \sigma}, U^{j + \sigma}).
\end{equation*}

Using Lemmas \ref{lemma2.4}-\ref{lemma2.5} and noticing
$(\tilde{\Lambda} U^{j + \sigma}, U^{j + \sigma}) \leq 0$,
one obtains
\begin{equation}
\frac{1}{4 \tau} (E^{j + 1} - E^j) + \frac{1}{2} \Delta_{0,t_{j + \sigma}}^{\alpha} \left \| U \right \|^2
\leq (f^{j + \sigma}, U^{j + \sigma}).
\label{eq3.8}
\end{equation}

\textbf{Step 1.} When $j = 0$.
From the inequality\/ \eqref{eq3.8}, it has
\begin{equation*}
\frac{1}{4 \tau} (E^{1} - E^0)
+ \frac{1}{2 \Gamma(2 - \alpha)} a_0^{(\alpha, \sigma)} ( \| U^1 \|^2 - \| U^0 \|^2)
\leq (f^{\sigma}, U^{\sigma}).
\end{equation*}
With the aid of virtue Cauchy-Schwarz inequality, we arrive at
\begin{align*}
\| U^1 \|^2 + \frac{2 \tau \sigma}{T^{\alpha} \Gamma(2 - \alpha)} a_0^{(\alpha, \sigma)} \| U^1 \|^2
\leq & \frac{2 \tau \sigma}{T^{\alpha} \Gamma(2 - \alpha)} a_0^{(\alpha, \sigma)} \| U^0 \|^2
+ 2 \sigma \| U^0 \|^2 + \frac{\tau \sigma}{\varepsilon_3} \| f^{\sigma} \|^2 \\
&\qquad + 8 \tau \sigma \varepsilon_3 (\| U^1 \|^2 + \| U^0 \|^2) , ~ \varepsilon_3 > 0,
\end{align*}
Let $\varepsilon_3 = \frac{1}{4 T^{\alpha} \Gamma(2 - \alpha)} a_0^{(\alpha, \sigma)}$,
it gives immediately the estimate for $u^1$, that is
\begin{align*}
\| U^1 \|^2
& \leq
\Big( \frac{4 \tau \sigma}{T^{\alpha} \Gamma(2 - \alpha) a_0^{(\alpha, \sigma)}} + 2 \sigma \Big) \| U^0 \|^2
+ \frac{4 \tau \sigma T^{\alpha} \Gamma(2 - \alpha)}{a_0^{(\alpha, \sigma)}} \| f^{\sigma} \|^2 \\
& \leq
\Big( \frac{4 T^{1 - \alpha} \sigma^{\alpha}}{\Gamma(2 - \alpha)} + 2 \sigma \Big) \| U^0 \|^2
+ 4 \sigma^{\alpha} T^{1 + \alpha} \Gamma(2 - \alpha) \| f^{\sigma} \|^2.
\end{align*}

\textbf{Step 2.} When $j \geq 1$, summing up for $j$ in\/ \eqref{eq3.8} from $1$ to $k - 1$
and doing some simple manipulations, it results
\begin{equation}
\begin{split}
\frac{1}{4 \tau} (E^{k} & - E^1)
+ \frac{1}{2 \tau^{\alpha} \Gamma(2 - \alpha)}
\left[ c_0^{(\alpha, \sigma)} \sum_{j = 1}^{k - 1} \| U^{j + 1} \|^2
- \sum_{j = 1}^{k - 1} \sum_{s = 2}^{j} ( c_{j - s}^{(\alpha, \sigma)}
- c_{j - s + 1}^{(\alpha, \sigma)} ) \| U^{s} \|^2 \right] \\
& \leq \frac{1}{2 \tau^{\alpha} \Gamma(2 - \alpha)} \| U^{0} \|^2 \sum_{j = 1}^{k - 1} c_{j}^{(\alpha, \sigma)}
+ \frac{1}{2 \tau^{\alpha} \Gamma(2 - \alpha)} \| U^{1} \|^2
\sum_{j = 1}^{k - 1} ( c_{j - 1}^{(\alpha, \sigma)} - c_{j}^{(\alpha, \sigma)} ) \\
&\qquad + \sum_{j = 1}^{k - 1} \| f^{j + \sigma} \| \cdot
\left[ \sigma \| U^{j + 1} \| + (1 - \sigma) \| U^{j} \| \right].
\end{split}
\label{eq3.9}
\end{equation}
To estimate the second term on the left hand side of inequality\/ \eqref{eq3.9},
Lemma \ref{lemma2.2} is applied.
Then
\begin{equation*}
\begin{split}
& \frac{1}{2 \tau^{\alpha} \Gamma(2 - \alpha)}
\left[ c_0^{(\alpha, \sigma)} \sum_{j = 1}^{k - 1} \| U^{j + 1} \|^2
- \sum_{j = 1}^{k - 1} \sum_{s = 2}^{j} ( c_{j - s}^{(\alpha, \sigma)}
- c_{j - s + 1}^{(\alpha, \sigma)} ) \| U^{s} \|^2 \right] \\
&\quad = \frac{1}{2 \tau^{\alpha} \Gamma(2 - \alpha)}
\sum_{j = 2}^{k} c_{k - j}^{(\alpha, \sigma)} \| U^{j} \|^2 \\
&\quad \geq
\frac{1}{2 \tau^{\alpha} \Gamma(2 - \alpha)} \frac{1 - \alpha}{2}
(j - 1 + \sigma)^{- \alpha} \sum_{j = 2}^{k} \| U^{j} \|^2 \\
&\quad \geq
\frac{1}{4 T^{\alpha} \Gamma(1 - \alpha)} \sum_{j = 2}^{k} \| U^{j} \|^2.
\end{split}
\end{equation*}
Bringing above estimate to inequality\/ \eqref{eq3.9} gets
\begin{align}
& \| U^{k} \|^2 + \frac{\tau \sigma}{T^{\alpha} \Gamma(1 - \alpha)} \sum_{j = 2}^{k} \| U^{j} \|^2
\nonumber \\
&~~ \leq \sigma E^1
+ \frac{2 \tau^{1 - \alpha} \sigma (4 - 3 \alpha) (k - 1 + \sigma)^{1 - \alpha}}{\Gamma(3 - \alpha)} \| U^{1} \|^2
+ \frac{2 \tau^{1 - \alpha} \sigma (k - 1 + \sigma)^{1 - \alpha}}{\Gamma(2 - \alpha)} \| U^{0} \|^2
\nonumber \\
&\qquad\qquad + 4 \tau \sigma \varepsilon_4 \sum_{j = 1}^{k - 1} \left( \sigma \| U^{j + 1} \| + (1 - \sigma) \| U^{j} \| \right)^2
+ \frac{\tau \sigma}{\varepsilon_4} \sum_{j = 1}^{k - 1} \| f^{j + \sigma} \|^2
\nonumber \\
&~~ \leq \sigma \left[ (4 \sigma^2 + 4 \sigma - 1) \| U^{1} \|^2 + 4 \sigma^2 \| U^{0} \|^2 \right]
+ \frac{2 \sigma (4 - 3 \alpha) T^{1 - \alpha}}{\Gamma(3 - \alpha)} \| U^{1} \|^2
\nonumber \\
&\qquad\qquad + \frac{2 \sigma T^{1 - \alpha}}{\Gamma(2 - \alpha)} \| U^{0} \|^2
+ 8 \tau \sigma \varepsilon_4 \sum_{j = 1}^{k} \| U^{j} \|^2
+ \frac{\tau \sigma}{\varepsilon_4} \sum_{j = 1}^{k - 1} \| f^{j + \sigma} \|^2, ~ \varepsilon_4 > 0,
\label{eq3.10}
\end{align}
where $E^1 \leq (4 \sigma^2 + 4 \sigma - 1) \| U^1 \|^2 + 4 \sigma^2 \| U^0 \|^2$.
Taking $\varepsilon_4 = \frac{1}{8 T^{\alpha} \Gamma(1 - \alpha)}$,  inequality\/ \eqref{eq3.10} leads to
\begin{align*}
\| U^{k} \|^2
& \leq
\left[ \frac{\tau \sigma}{T^{\alpha} \Gamma(1 - \alpha)} + \frac{2 \sigma  (4 - 3 \alpha) T^{1 - \alpha}}{\Gamma(3 - \alpha)}
 + 4 \sigma^3 + 4 \sigma^2 - \sigma \right] \| U^{1} \|^2 \\
&\qquad + \left( \frac{2 \sigma T^{1 - \alpha}}{\Gamma(2 - \alpha)}
+ 4 \sigma^3 \right) \| U^{0} \|^2 + 8 \tau \sigma T^{\alpha} \Gamma(1 - \alpha) \sum_{j = 1}^{k - 1} \| f^{j + \sigma} \|^2 \\
& \leq
\left[ \frac{T^{1 - \alpha} \sigma}{\Gamma(1 - \alpha)} + \frac{2 \sigma (4 - 3 \alpha) T^{1 - \alpha}}{\Gamma(3 - \alpha)}
 + 4 \sigma^3 + 4 \sigma^2 - \sigma \right] \| U^{1} \|^2 \\
&\qquad + \left( \frac{2 \sigma T^{1 - \alpha}}{\Gamma(2 - \alpha)}
+ 4 \sigma^3 \right) \| U^{0} \|^2 + 8 \sigma T^{1 + \alpha} \Gamma(1 - \alpha) \sum_{j = 1}^{k - 1} \| f^{j + \sigma} \|^2.
\end{align*}
Hence, the targeted results are immediately completed. \hfill $\Box$

Next, the convergence of\/ \eqref{eq3.5} is discussed.
\begin{theorem}
Assume $u(x, y, t)$ be the sufficiently smooth exact solution of\/ \eqref{eq3.1}-\eqref{eq3.4},
$\{ u_{i l}^j \mid x_i \in \hat{\omega}, ~0 \leq j \leq M \}$ be the solution
of the problem\/ \eqref{eq3.5}. Let
$\xi_{i l}^j = u(x_i, y_l, t_j) - u_{i l}^j  ~(0 \leq i \leq N_x, ~ 0 \leq l \leq N_y, ~ 0 \leq j \leq M)$.
Then, for $j = 0, 1, 2, \cdots, M$, we have
\begin{equation*}
\| \xi^j \| \leq C_3 (\tau^2 + h_{x}^2 + h_{y}^2), ~ 0 \leq j \leq M,
\end{equation*}
where $C_3$ is a positive constant, which may depend on $\alpha$ and $T$.
\label{th3.2}
\end{theorem}
\textbf{Proof.}
Subtracting\/ \eqref{eq3.5} from\/ \eqref{eq3.1}-\eqref{eq3.4},
the error equations are
\begin{equation*}
\begin{cases}
\delta_{\hat{t}} \xi_{i l}^{j} + \Delta^{\alpha}_{0,t_{j + \sigma}} \xi_{i l}
= \tilde{\Lambda} \xi_{i l}^{j + \sigma}
 + \tilde{R}_{i l}^{j + \sigma},  & 1 \leq i \leq N_x - 1,~ 1 \leq l \leq N_y - 1, ~ 0 \leq j \leq M-1, \\
\xi_{i l}^{-1} = \xi_{i l}^{0} = 0,  & 0 \leq i \leq N_x, ~ 0 \leq l \leq N_y, \\
\xi_{0 l}^{j} = \xi_{N_x l}^{j} = 0,  & 1 \leq l \leq N_y, ~ 0 \leq j \leq M, \\
\xi_{i 0}^{j} = \xi_{i N_y}^{j} = 0, & 0 \leq i \leq N_x, ~ 0 \leq j \leq M,
\end{cases}
\end{equation*}
with $\tilde{R}_{i l} = \mathcal{O}(\tau^2 + h_{x}^2 + h_{y}^2)$.
After that, the following procedure is similar to Theorem \ref{th3.1}, and the error $\xi^j$ yields
\begin{equation*}
\| \xi^j \|_{2D} \leq C_3 (\tau^2 + h_{x}^2 + h_{y}^2), ~ 0 \leq j \leq M,
\end{equation*}
where $C_3$ is a positive constant, which may depend on $\alpha$ and $T$.
\hfill $\Box$

Similar to the 1D case, if the solution of Eq. \eqref{eq3.1} is non-smooth,
several alternative approaches \cite{zeng2017nonsmooth,liao2018nonsmooth} can be used to address this problem.
Some additional symbols are needed for presentation of the equivalent matrix form of the IDS\/ \eqref{eq3.5}.
\begin{equation*}
\hat{U}_{l}^{j} = [u_{1 l}^{j}, u_{2 l}^{j}, \cdots, u_{N_x - 1, l}^{j}], \quad
\tilde{U}^{j} = [\hat{U}_{1}^{j}, \hat{U}_{2}^{j}, \cdots, \hat{U}_{N_y - 1}^{j}]^{T},
\end{equation*}
\begin{equation*}
\hat{u}_{l} = \left[ u_{0}(x_1, y_l), u_{0}(x_2, y_l), \cdots, u_{0}(x_{N_{x} - 2}, y_l), u_{0}(x_{N_{x} - 1}, y_l) \right]^{T},
\end{equation*}
\begin{equation*}
\hat{f}_{l}^{j + \sigma} = \left[ f_{1 l}^{j + \sigma}, f_{2 l}^{j + \sigma}, \cdots, f_{N_x - 1, l}^{j + \sigma} \right]^{T}, \quad
\tilde{f}^{j + \sigma} = \left[ \hat{f}_{1}^{j + \sigma}, \hat{f}_{2}^{j + \sigma}, \cdots, \hat{f}_{N_y - 1}^{j + \sigma} \right]^{T},
\end{equation*}
\begin{align*}
\hat{A}_{l}^{j + \sigma} = & -\textrm{diag} \left( \left[ (d_{1 - \frac{1}{2}, l}^{j + \sigma} + d_{1 + \frac{1}{2}, l}^{j + \sigma}),
(d_{2 - \frac{1}{2}, l}^{j + \sigma} + d_{2 + \frac{1}{2}, l}^{j + \sigma}), \cdots,
(d_{N_{x} - 1 - \frac{1}{2}, l}^{j + \sigma}+ d_{N_{x} - 1 + \frac{1}{2}, l}^{j + \sigma}) \right] \right) \\
& + \textrm{diag} \left( \left[ d_{2 - \frac{1}{2}, l}^{j + \sigma}, d_{3 - \frac{1}{2}, l}^{j + \sigma}, \cdots,
d_{N_{x} - 1 - \frac{1}{2}, l}^{j + \sigma} \right], -1 \right)
+ \textrm{diag} \left( \left[ d_{1 + \frac{1}{2}, l}^{j + \sigma}, d_{2 + \frac{1}{2}, l}^{j + \sigma}, \cdots,
d_{N_{x} - 2 + \frac{1}{2}, l}^{j + \sigma} \right],1 \right),
\end{align*}
\begin{equation*}
\bar{B}_{l}^{j + \sigma} = - \textrm{diag} \left( (k_{1,l + 1 - \frac{1}{2}}^{j + \sigma} + k_{1,l - \frac{1}{2}}^{j + \sigma}),
(k_{2,l + 1 - \frac{1}{2}}^{j + \sigma} + k_{2,l - \frac{1}{2}}^{j + \sigma}), \cdots,
(k_{N_{x} - 1,l + 1 - \frac{1}{2}}^{j + \sigma} + k_{N_{x} - 1,l - \frac{1}{2}}^{j + \sigma}) \right),
\end{equation*}
\begin{equation*}
\hat{B}_{s}^{j + \sigma} = \textrm{diag} \left( k_{1, s - \frac{1}{2}}^{j + \sigma}, k_{2, s - \frac{1}{2}}^{j + \sigma},
\cdots, k_{N_{x} - 1, s - \frac{1}{2}}^{j + \sigma} \right), \quad
\hat{Q}_{l}^{j + \sigma} = \textrm{diag} \left( q_{1,l}^{j + \sigma}, q_{2,l}^{j + \sigma},
\cdots, q_{N_{x} - 1,l}^{j + \sigma} \right),
\end{equation*}
\begin{equation*}
\hat{\xi}_{l}^{j + \sigma} = \left[
d_{1 - \frac{1}{2},1}^{j + \sigma} \left[ \sigma u_{0, l}^{j + 1} + (1 - \sigma) u_{0, l}^{j} \right],
0, \cdots, 0,
d_{N_{x} - \frac{1}{2},1}^{j + \sigma} \left[ \sigma u_{N_{x}, l}^{j + 1} + (1 - \sigma) u_{N_{x}, l}^{j} \right]
\right]_{(N_{x} - 1) \times 1}^{T}
\end{equation*}
and
\begin{equation*}
v_1^{j + \sigma} = \begin{bmatrix}
k_{1, 1 - \frac{1}{2}}^{j + \sigma} \left[ \sigma u_{1,0}^{j + 1} + (1 - \sigma) u_{1,0}^{j} \right] \\
k_{2, 1 - \frac{1}{2}}^{j + \sigma} \left[ \sigma u_{2,0}^{j + 1} + (1 - \sigma) u_{2,0}^{j} \right] \\
\vdots \\
k_{N_{x} - 1, 1 - \frac{1}{2}}^{j + \sigma} \left[ \sigma u_{N_{x} - 1,0}^{j + 1} + (1 - \sigma) u_{N_{x} - 1,0}^{j} \right]
\end{bmatrix},
\end{equation*}
\begin{equation*}
v_2^{j + \sigma} = \begin{bmatrix}
k_{1, N_{y} - \frac{1}{2}}^{j + \sigma} \left[ \sigma u_{1, N_{y}}^{j + 1} + (1 - \sigma) u_{1, N_{y}}^{j} \right] \\
k_{2, N_{y} - \frac{1}{2}}^{j + \sigma} \left[ \sigma u_{2, N_{y}}^{j + 1} + (1 - \sigma) u_{2, N_{y}}^{j} \right] \\
\vdots \\
k_{N_{x} - 1, N_{y} - \frac{1}{2}}^{j + \sigma} \left[ \sigma u_{N_{x} - 1, N_{y}}^{j + 1}
+ (1 - \sigma) u_{N_{x} - 1, N_{y}}^{j} \right]
\end{bmatrix}.
\end{equation*}

Finally, the equivalent matrix form of\/ \eqref{eq3.5} below is derived to complete this section.
\begin{equation}
\begin{cases}
\begin{split}
\mathcal{S}^{j + 1} \tilde{U}^{j + 1} = P^{j} \tilde{U}^{j} & - (2 \sigma - 1) h_{x}^2 h_{y}^2 \tilde{U}^{j - 1}
- \frac{2 \tau^{1 - \alpha} h_{x}^2 h_{y}^2}{\Gamma(2 - \alpha)}
\sum_{s = 0}^{j - 1} c_{j - s}^{(\alpha, \sigma)} \left( \tilde{U}^{s + 1} - \tilde{U}^{s} \right) \\
& + 2 \tau h_{x}^2 h_{y}^2 \tilde{f}^{j + \sigma} + 2 \tau h_{x}^2 h_{y}^2 (\tilde{\xi}_1^{j + \sigma} + \tilde{\xi}_2^{j + \sigma}),
\quad 0 \leq j \leq M - 1,
\end{split} \\
\tilde{U}^{0} = \tilde{u}_0,
\end{cases}
\label{eq3.11}
\end{equation}
in which
\begin{equation*}
\mathcal{S}^{j + 1} = \left[ (2 \sigma + 1) h_{x}^2 h_{y}^2
+ \frac{2 \tau^{1 - \alpha} h_{x}^2 h_{y}^2}{\Gamma(2 - \alpha)} c_{0}^{(\alpha, \sigma)} \right] I
- 2 \tau \sigma \left( h_{y}^2 \tilde{A}^{j + \sigma} + h_{x}^2 \tilde{B}^{j + \sigma}
- h_{x}^2 h_{y}^2 \tilde{Q}^{j + \sigma} \right),
\end{equation*}
\begin{equation*}
P^{j} = \left[ 4 \sigma h_{x}^2 h_{y}^2
+ \frac{2 \tau^{1 - \alpha} h_{x}^2 h_{y}^2}{\Gamma(2 - \alpha)} c_{0}^{(\alpha, \sigma)} \right] I
+ 2 \tau (1 - \sigma) \left( h_{y}^2 \tilde{A}^{j + \sigma} + h_{x}^2 \tilde{B}^{j + \sigma}
- h_{x}^2 h_{y}^2 \tilde{Q}^{j + \sigma} \right)
\end{equation*}
and
\begin{equation*}
\tilde{u}_0 = \begin{bmatrix}
\hat{u}_1 \\
\hat{u}_2 \\
\vdots \\
\hat{u}_{N_{y} - 2} \\
\hat{u}_{N_{y} - 1}
\end{bmatrix}, \quad
\tilde{\xi}_1^{j + \sigma} = \begin{bmatrix}
\hat{\xi}_1 \\
\hat{\xi}_2 \\
\vdots \\
\hat{\xi}_{N_{y} - 2} \\
\hat{\xi}_{N_{y} - 1}
\end{bmatrix}, \quad
\tilde{\xi}_2^{j + \sigma} = \begin{bmatrix}
v_1^{j + \sigma} \\
0 \\
\vdots \\
0 \\
v_2^{j + \sigma}
\end{bmatrix}_{(N_{x} - 1) \times (N_{y} - 1)},
\end{equation*}
whereas
\begin{equation*}
\tilde{A}^{j + \sigma} = \textrm{diag} \left( \hat{A}_{1}^{j + \sigma}, \hat{A}_{2}^{j + \sigma}, \cdots,
\hat{A}_{N_{y} - 1}^{j + \sigma} \right), ~
\tilde{Q}^{j + \sigma} = \textrm{diag} \left( \hat{Q}_{1}^{j + \sigma}, \hat{Q}_{2}^{j + \sigma}, \cdots,
\hat{Q}_{N_{y} - 1}^{j + \sigma} \right),
\end{equation*}
\begin{align*}
\tilde{B}^{j + \sigma} = & \textrm{diag} \left( \left[ \bar{B}_{1}^{j + \sigma}, \bar{B}_{2}^{j + \sigma}, \cdots,
\bar{B}_{N_{y} - 1}^{j + \sigma} \right] \right) +
\textrm{diag} \left( \left[ \hat{B}_{2}^{j + \sigma}, \hat{B}_{3}^{j + \sigma},
 \cdots, \hat{B}_{N_{y} - 1}^{j + \sigma} \right], -1 \right) \\
&\quad + \textrm{diag} \left( \left[ \hat{B}_{2}^{j + \sigma}, \hat{B}_{3}^{j + \sigma},
\cdots, \hat{B}_{N_{y} - 1}^{j + \sigma} \right], 1 \right).
\end{align*}
Investigation on the expression of $\mathcal{S}^{j + 1}$,
it can be found that the coefficient matrix $\mathcal{S}^{j + 1}$
is a large sparse banded symmetric matrix.
For the sake of clarity, Fig. \ref{fig1} is an example of $\mathcal{S}^{j + 1}$ corresponding to
$h_{x} = h_{y} = \frac{1}{8}$ and $\tau = \frac{1}{5}$.
\begin{figure}[H]
\centering
\includegraphics[width=2.8in,height=2.3in]{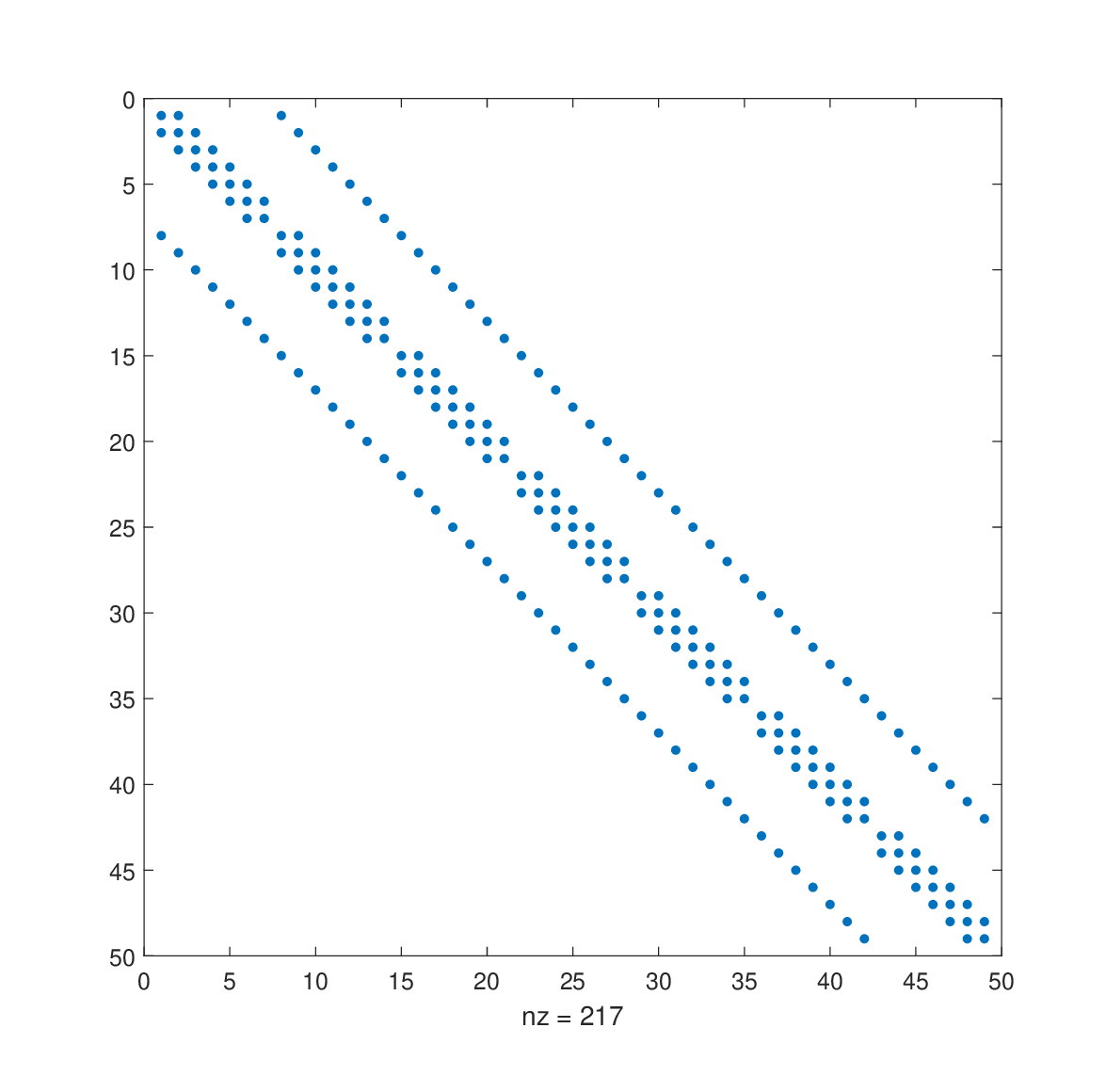}
\includegraphics[width=2.8in,height=2.3in]{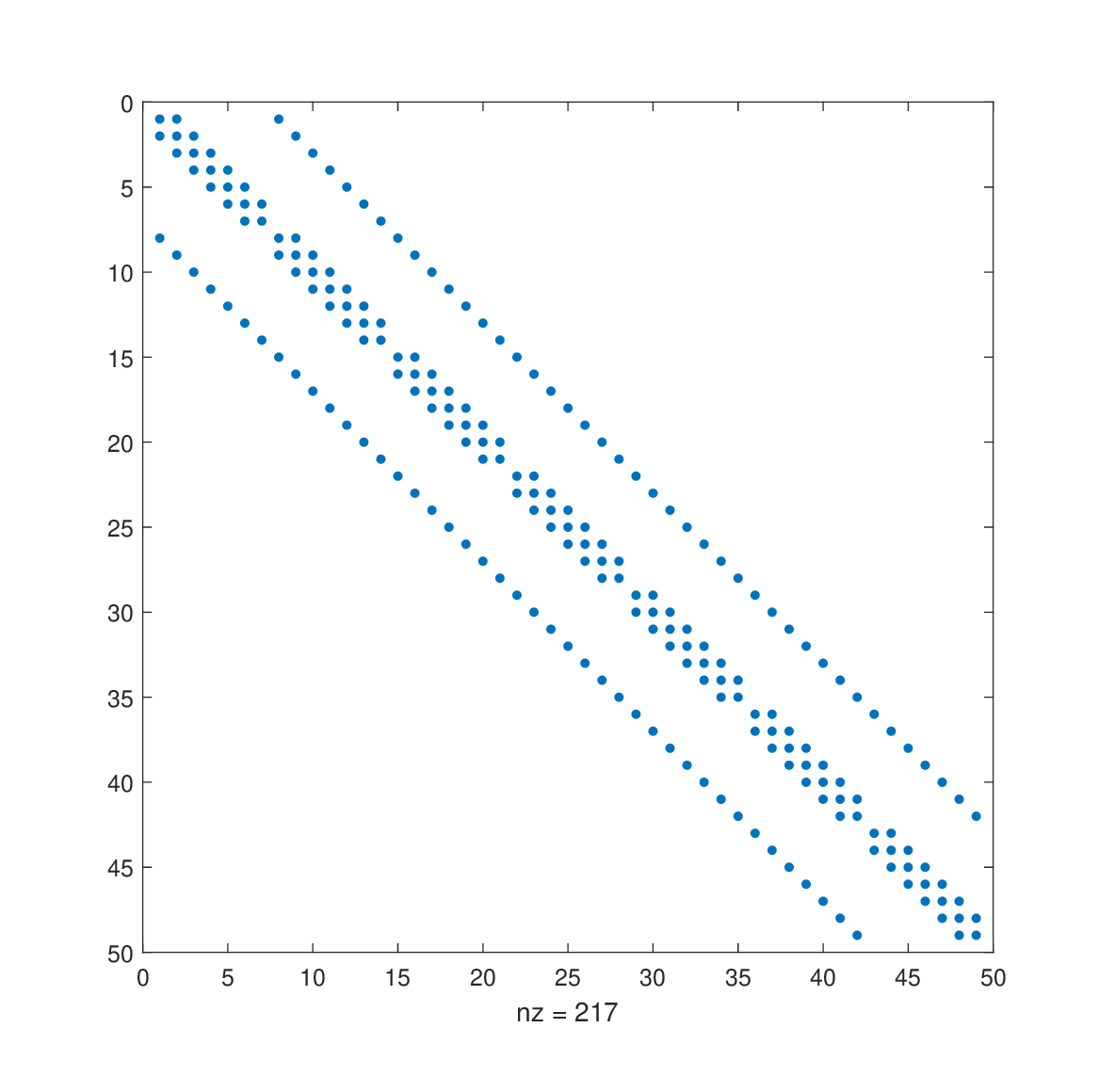}
\caption{Sparsity pattern of $\mathcal{S}^{j + 1}$ with $h_{x} = h_{y} = \frac{1}{8}$
and $\tau = \frac{1}{5}$. Left: $j = 0$; Right: $j = 1$.}
\label{fig1}
\end{figure}

\section{Numerical results}
\label{sec4}

Numerical results are provided to validate the error estimates
obtained in Theorems \ref{th2.2} and \ref{th3.2}
for the proposed difference schemes\/ \eqref{eq2.3} and\/ \eqref{eq3.5}, respectively,
which are given in Examples 1-2.
Moreover, the method proposed in \cite{gao2017second} (denote as Gao's method)
is employed to solve\/ \eqref{eq1.1}, and the corresponding errors also reported in Examples 1-2.
In Example 3, three preconditioned iterative methods are employed for solving the linear system of the
two-dimensional case.
For simplicity, when test Examples 2, we take $h_{x} = h_{y} = h$ in this manuscript, and
let
\begin{equation*}
Error_1(h, \tau) = \max_{0 \leq j \leq M} \| e^{j} \|_\infty, \quad
Error_2(h, \tau) = \max_{0 \leq j \leq M} \| e^{j} \|,
\end{equation*}
\begin{equation*}
Error_3(h, \tau) = \max_{0 \leq j \leq M} \| \xi^{j} \|_\infty, \quad
Error_4(h, \tau) = \max_{0 \leq j \leq M} \| \xi^{j} \|,
\end{equation*}
\begin{equation*}
rate1_{\tau} =\log_{\tau_1/ \tau_2} \frac{Error_1(h, \tau_1)}{Error_1(h, \tau_2)}, \quad
rate1_{h} =\log_{h_1/ h_2} \frac{Error_1(h_1, \tau)}{Error_1(h_2, \tau)},
\end{equation*}
\begin{equation*}
rate2_{\tau} =\log_{\tau_1/ \tau_2} \frac{Error_2(h, \tau_1)}{Error_2(h, \tau_2)}, \quad
rate2_{h} =\log_{h_1/ h_2} \frac{Error_2(h_1, \tau)}{Error_2(h_2, \tau)},
\end{equation*}
\begin{equation*}
rate3_{\tau} =\log_{\tau_1/ \tau_2} \frac{Error_3(h, \tau_1)}{Error_3(h, \tau_2)}, \quad
rate3_{h} =\log_{h_1/ h_2} \frac{Error_3(h_1, \tau)}{Error_3(h_2, \tau)},
\end{equation*}
\begin{equation*}
rate4_{\tau} =\log_{\tau_1/ \tau_2} \frac{Error_4(h, \tau_1)}{Error_4(h, \tau_2)}, \quad
rate4_{h} =\log_{h_1/ h_2} \frac{Error_4(h_1, \tau)}{Error_4(h_2, \tau)}.
\end{equation*}

All experiments were performed on a Windows 10 (64 bit) desktop-Intel(R) Xeon(R) E5504
CPU 2.00GHz 2.00GHz (two processors), 48GB of RAM using MATLAB R2015b.
\begin{table}[H]\footnotesize \tabcolsep=3.5pt
\caption{$L_2$-norm and maximum norm errors and convergence orders for Example 1 where $h = 1/2000$.}
\centering
\begin{tabular}{cccccccccc}
\hline
& & \multicolumn{4}{c}{Our method} & \multicolumn{4}{c}{Gao's method} \\
[-2pt] \cmidrule(lr){3-6} \cmidrule(lr){7-10} \\ [-10pt]
$\alpha$ & $\tau$ & $Error_1(h, \tau)$ & $rate1_{\tau}$ & $Error_2(h, \tau)$ & $rate2_{\tau}$
& $Error_1(h, \tau)$ & $rate1_{\tau}$ & $Error_2(h, \tau)$ & $rate2_{\tau}$ \\
\hline
0.10 & 1/8 & 6.9433e-05 & -- & 5.0972e-05 & -- & 7.0252e-04 & -- & 4.4325e-04 & -- \\
     & 1/16 & 1.8594e-05 & 1.9007 & 1.3514e-05 & 1.9152 & 1.7601e-04 & 1.9969 & 1.1105e-04 & 1.9969 \\
     & 1/32 & 4.7487e-06 & 1.9693 & 3.4269e-06 & 1.9795 & 4.3992e-05 & 2.0003 & 2.7752e-05 & 2.0005 \\
     & 1/64 & 1.1958e-06 & 1.9896 & 8.6184e-07 & 1.9914 & 1.0966e-05 & 2.0042 & 6.9152e-06 & 2.0047 \\
     & 1/128 & 3.0034e-07 & 1.9933 & 2.1639e-07 & 1.9938 & 2.7094e-06 & 2.0170 & 1.7057e-06 & 2.0194 \\
\hline
0.50  & 1/8 & 5.5452e-04 & -- & 3.4264e-04 & -- & 9.0968e-04 & -- & 5.7256e-04 & -- \\
     & 1/16 & 1.4046e-04 & 1.9810 & 8.6759e-05 & 1.9816 & 2.2859e-04 & 1.9926 & 1.4385e-04 & 1.9929 \\
     & 1/32 & 3.5268e-05 & 1.9938 & 2.1774e-05 & 1.9944 & 5.7192e-05 & 1.9989 & 3.5981e-05 & 1.9993 \\
     & 1/64 & 8.8017e-06 & 2.0025 & 5.4299e-06 & 2.0036 & 1.4267e-05 & 2.0031 & 8.9715e-06 & 2.0038 \\
     & 1/128 & 2.1711e-06 & 2.0194 & 1.3362e-06 & 2.0227 & 3.5351e-06 & 2.0129 & 2.2201e-06 & 2.0147 \\
\hline
0.90  & 1/8 & 1.1057e-03 & -- & 6.9253e-04 & -- & 1.1521e-03 & -- & 7.2351e-04 & -- \\
     & 1/16 & 2.7756e-04 & 1.9941 & 1.7381e-04 & 1.9944 & 2.8877e-04 & 1.9963 & 1.8131e-04 & 1.9966 \\
     & 1/32 & 6.9391e-05 & 2.0000 & 4.3440e-05 & 2.0004 & 7.2144e-05 & 2.0010 & 4.5288e-05 & 2.0013 \\
     & 1/64 & 1.7304e-05 & 2.0036 & 1.0828e-05 & 2.0043 & 1.7987e-05 & 2.0039 & 1.1286e-05 & 2.0046 \\
     & 1/128 & 4.2930e-06 & 2.0110 & 2.6833e-06 & 2.0127 & 4.4632e-06 & 2.0108 & 2.7976e-06 & 2.0123 \\
\hline
0.99  & 1/8 & 1.2116e-03 & -- & 7.6041e-04 & -- & 1.2156e-03 & -- & 7.6310e-04 & -- \\
     & 1/16 & 3.0380e-04 & 1.9958 &1.9066e-04  & 1.9958 & 3.0475e-04 & 1.9960 & 1.9131e-04 & 1.9960 \\
     & 1/32 & 7.5972e-05 & 1.9996 & 4.7675e-05 & 1.9997 & 7.6204e-05 & 1.9997 & 4.7834e-05 & 1.9998 \\
     & 1/64 & 1.8967e-05 & 2.0020 & 1.1900e-05 & 2.0023 & 1.9024e-05 & 2.0020 & 1.1939e-05 & 2.0024 \\
     & 1/128 & 4.7154e-06 & 2.0081 & 2.9559e-06 & 2.0093 & 4.7296e-06 & 2.0080 & 2.9656e-06 & 2.0093 \\
\hline
\end{tabular}
\label{tab1}
\end{table}

\subsection{The 1D case}
\label{sec4.1}
At first, the 1D TFRDE with zero boundary condition is considered.

\noindent{\textbf{Example 1.}} In this example, we consider the Eq. \eqref{eq1.1}
on space interval $[0, L] = [0, 1]$ and time interval $[0, T] = [0, 1]$ with
the coefficients $k(x,t) = x \exp(-t) + 1$, $q(x,t) = t^2 \cos(x)$, and the source term
\begin{align*}
f(x,t) = & x^2 (1 - x)^2 \left[ (3 + \alpha) t^{2 + \alpha}
+ \frac{\Gamma(4 + \alpha)}{\Gamma(4)} t^3 \right]
- t^{3 + \alpha} \Big\{ \Big[ \exp(-t) \left( 16 x^3 - 18 x^2 + 4 x \right) \\
&\quad - \left( 12 x^2 - 12 x + 2 \right) \Big]
- t^2 \cos(x) x^2 (1 - x)^2 \Big\}.
\end{align*}
For the above values, the exact solution is $u(x, t) = t^{3 + \alpha} x^2 (1 - x)^2$.

Firstly, fixing the spatial step $h = 1/2000$ and taking different temporal steps.
Table \ref{tab1} displays the maximum norm errors, $L_2$-norm errors and temporal convergence orders
of the IDS\/ \eqref{eq2.3} for $\alpha= 0.1, 0.5, 0.9, 0.99$. It shows that the convergence order of the scheme
in temporal direction is $\mathcal{O}(\tau^2)$.
It is in accord with the theoretical result in Section \ref{sec2.2}.
Although the temporal convergence orders of the proposed method are smaller than the Gao's method,
the errors of the proposed method are slightly better than the Gao's method.
Afterwards, we investigate the spatial convergence rate for a fixed temporal step size $\tau = h$.
Table \ref{tab2} lists the maximum norm errors, $L_2$-norm errors
and spatial convergence rates of the scheme\/ \eqref{eq2.3}.
From Table \ref{tab2}, the errors of the Gao's method are smaller than our method.
However, the spatial convergence orders of our method are slightly better than the Gao's method.
As predicted by the theoretical estimates, the temporal and spatial approximation orders
of our proposed scheme\/ \eqref{eq2.3} are close to 2, i.e.,
the slopes of the error curves in Fig. \ref{fig2} is 2, for $\alpha= 0.1, 0.5, 0.9, 0.99$.
\begin{figure}[t]
\centering
\includegraphics[width=2.4in,height=2.4in]{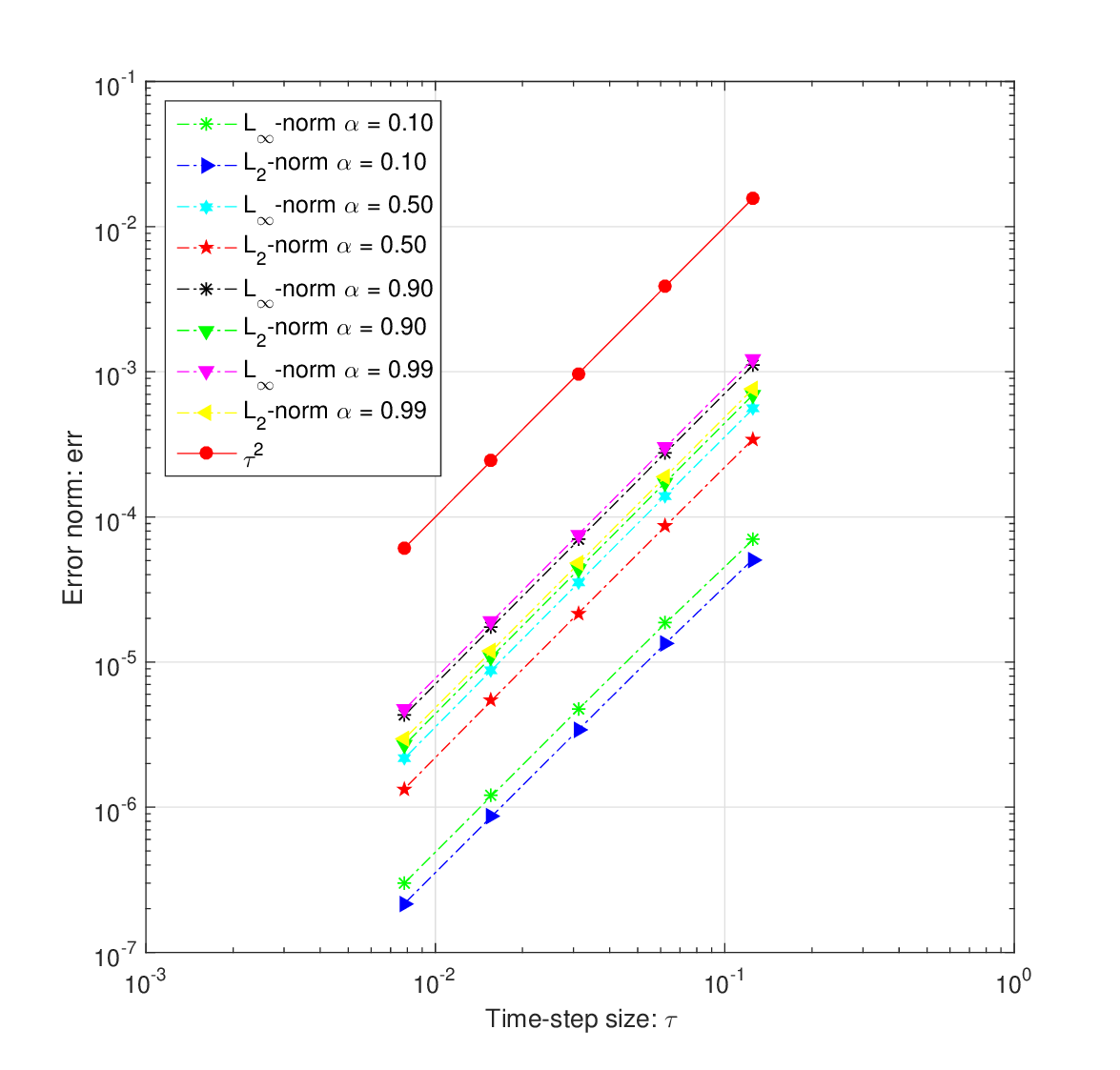}
\includegraphics[width=2.4in,height=2.4in]{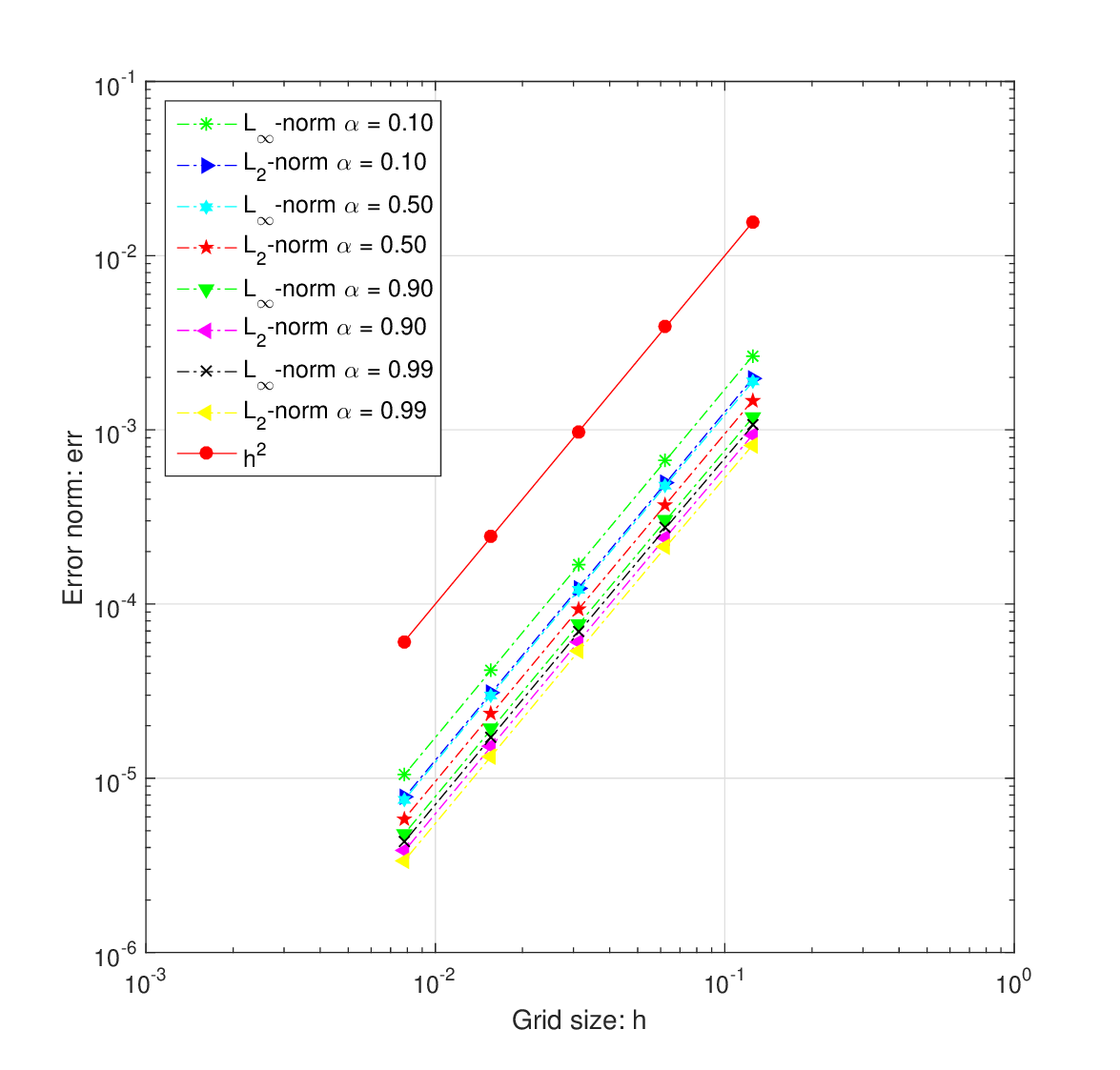}
\caption{Comparison the order of accuracy obtained by our proposed IDS for Example 1 in
space and time variables. Left: time direction; Right: space direction.}
\label{fig2}
\end{figure}
\begin{table}[!tbp]\footnotesize \tabcolsep=3.5pt
\caption{$L_2$-norm and maximum norm errors and convergence orders for Example 1 where $\tau = h$.}
\centering
\begin{tabular}{cccccccccc}
\hline
& & \multicolumn{4}{c}{Our method} & \multicolumn{4}{c}{Gao's method} \\
[-2pt] \cmidrule(lr){3-6} \cmidrule(lr){7-10} \\ [-10pt]
$\alpha$ & $h$ & $Error_1(h, \tau)$ & $rate1_{h}$ & $Error_2(h, \tau)$ & $rate2_{h}$
& $Error_1(h, \tau)$ & $rate1_{h}$ & $Error_2(h, \tau)$ & $rate2_{h}$ \\
\hline
0.10 & 1/8 & 2.6665e-03 & -- & 1.9792e-03 & -- & 1.9327e-03 & -- & 1.5130e-03 & -- \\
     & 1/16 & 6.6754e-04 & 1.9980 & 4.9576e-04 & 1.9972 & 4.9389e-04 & 1.9684 & 3.8359e-04 & 1.9798 \\
     & 1/32 & 1.6749e-04 & 1.9948 & 1.2398e-04 & 1.9995 & 1.2390e-04 & 1.9950 & 9.6228e-05 & 1.9950 \\
     & 1/64 & 4.1873e-05 & 2.0000 & 3.0997e-05 & 2.0000 & 3.1003e-05 & 1.9987 & 2.4078e-05 & 1.9987 \\
     & 1/128 & 1.0468e-05 & 2.0001 & 7.7491e-06 & 2.0000 & 7.7532e-06 & 1.9995 & 6.0210e-06 & 1.9996 \\
\hline
0.50 & 1/8 & 1.8911e-03 & -- & 1.4793e-03 & -- & 1.5625e-03 & -- & 1.2449e-03 & -- \\
     & 1/16 & 4.8047e-04 & 1.9767 & 3.7325e-04 & 1.9867 & 3.9788e-04 & 1.9735 & 3.1726e-04 & 1.9723 \\
     & 1/32 & 1.2030e-04 & 1.9978 & 9.3494e-05 & 1.9972 & 9.9931e-05 & 1.9933 & 7.9693e-05 & 1.9931 \\
     & 1/64 & 3.0083e-05 & 1.9996 & 2.3385e-05 & 1.9993 & 2.5017e-05 & 1.9980 & 1.9952e-05 & 1.9979  \\
     & 1/128 & 7.5230e-06 & 1.9996 & 5.8475e-06 & 1.9997 & 6.2582e-06 & 1.9991 & 4.9911e-06 & 1.9991 \\
\hline
0.90 & 1/8 & 1.1794e-03 & -- &9.3874e-04  & -- & 1.1490e-03 & -- & 9.0751e-04 & -- \\
     & 1/16 &3.0221e-04  &1.9645  & 2.4153e-04 & 1.9585 & 2.9512e-04 & 1.9610 & 2.3424e-04 & 1.9539 \\
     & 1/32 & 7.6410e-05 & 1.9837 & 6.0837e-05 & 1.9892 & 7.4468e-05 & 1.9866 & 5.9059e-05 & 1.9878 \\
     & 1/64 & 1.9150e-05 & 1.9964 & 1.5251e-05 & 1.9960 & 1.8680e-05 & 1.9951 & 1.4811e-05 & 1.9955 \\
     & 1/128 & 4.7945e-06 & 1.9979 & 3.8192e-06 & 1.9976 & 4.6776e-06 & 1.9977 & 3.7094e-06 & 1.9974 \\
\hline
0.99 & 1/8 & 1.0616e-03 & -- & 8.2033e-04 & -- & 1.0589e-03 & -- & 8.1763e-04 & -- \\
     & 1/16 & 2.7337e-04 & 1.9573 &2.1251e-04  & 1.9487 & 2.7275e-04 & 1.9569 & 2.1188e-04 & 1.9482 \\
     & 1/32 & 6.8854e-05 & 1.9892 & 5.3586e-05 & 1.9876 & 6.8703e-05 & 1.9891 & 5.3433e-05 & 1.9874 \\
     & 1/64 & 1.7248e-05 & 1.9971 &1.3427e-05  & 1.9967 & 1.7211e-05 & 1.9970 & 1.3389e-05 & 1.9967 \\
     & 1/128 & 4.3148e-06 & 1.9991 & 3.3594e-06 & 1.9989 & 4.3055e-06 & 1.9991 & 3.3500e-06 & 1.9988 \\
\hline
\end{tabular}
\label{tab2}
\end{table}

\begin{table}[th]\footnotesize \tabcolsep=3.5pt
\caption{$L_2$-norm and maximum norm errors and convergence orders
for Example 2 where $h_{x} = h_{y} = 1/1000$.}
\centering
\begin{tabular}{cccccccccc}
\hline
& & \multicolumn{4}{c}{Our method} & \multicolumn{4}{c}{Gao's method} \\
[-2pt] \cmidrule(lr){3-6} \cmidrule(lr){7-10} \\ [-10pt]
$\alpha$ & $\tau$ & $Error_3(h, \tau)$ & $rate3_{\tau}$ & $Error_4(h, \tau)$ & $rate4_{\tau}$
& $Error_3(h, \tau)$ & $rate3_{\tau}$ & $Error_4(h, \tau)$ & $rate4_{\tau}$ \\
\hline
0.10 & 1/5 & 1.1189e-05 & -- & 3.6846e-06 & -- & 2.2674e-04 & -- & 9.1250e-05 & -- \\
     & 1/10 & 2.7763e-06 & 2.0109 & 9.2335e-07 & 1.9966 & 5.7608e-05 & 1.9767 & 2.3175e-05 & 1.9773 \\
     & 1/20 & 6.7644e-07 & 2.0371 & 2.2821e-07 & 2.0166 & 1.4443e-05 & 1.9959 & 5.8084e-06 & 1.9964 \\
     & 1/40 & 1.5089e-07 & 2.1645 & 5.3935e-08 & 2.0810 & 3.5961e-06 & 2.0059 & 1.4444e-06 & 2.0077 \\
     & 1/80 & 3.1178e-08 & 2.2749 & 1.7958e-08 & 1.5866 & 8.8094e-07 & 2.0293 & 3.5206e-07 & 2.0366 \\
\hline
0.50 & 1/5 & 1.7332e-04 & -- &6.8422e-05  & -- & 2.7098e-04 & -- & 1.0877e-04 & -- \\
     & 1/10 &4.5279e-05  & 1.9365 & 1.7852e-05 & 1.9384 & 6.9498e-05 & 1.9631 & 2.7890e-05 & 1.9635 \\
     & 1/20 & 1.1519e-05 & 1.9748 & 4.5365e-06 & 1.9764 & 1.7501e-05 & 1.9895 & 7.0205e-06 & 1.9901 \\
     & 1/40 & 2.8857e-06 & 1.9970 & 1.1342e-06 & 2.0000 & 4.3687e-06 & 2.0022 & 1.7504e-06 & 2.0039 \\
     & 1/80 & 7.0598e-07 & 2.0313 & 2.7556e-07 & 2.0412 & 1.0748e-06 & 2.0231 & 4.2883e-07 & 2.0292 \\
\hline
0.90 & 1/5 & 3.1489e-04 & -- & 1.2578e-04 & -- & 3.2578e-04 & -- & 1.3042e-04 & -- \\
     & 1/10 & 8.0729e-05 & 1.9637 & 3.2235e-05 & 1.9642 & 8.3188e-05 & 1.9695 & 3.3299e-05 & 1.9696 \\
     & 1/20 & 2.0319e-05 & 1.9902 & 8.1098e-06 & 1.9909 & 2.0897e-05 & 1.9931 & 8.3617e-06 & 1.9936 \\
     & 1/40 & 5.0730e-06 & 2.0019 & 2.0226e-06 & 2.0035 & 5.2126e-06 & 2.0032 & 2.0837e-06 & 2.0046 \\
     & 1/80 & 1.2509e-06 & 2.0199 & 4.9687e-07 & 2.0253 & 1.2852e-06 & 2.0200 & 5.1193e-07 & 2.0251 \\
\hline
0.99 & 1/5 & 3.3967e-04 & -- & 1.3587e-04 & -- & 3.4054e-04 & -- & 1.3625e-04 & -- \\
     & 1/10 & 8.6609e-05 & 1.9715 & 3.4643e-05 & 1.9716 & 8.6796e-05 & 1.9721 & 3.4726e-05 & 1.9722 \\
     & 1/20 & 2.1744e-05 & 1.9939 & 8.6957e-06 & 1.9942 & 2.1787e-05 & 1.9942 & 8.7149e-06 & 1.9945 \\
     & 1/40 & 5.4253e-06 & 2.0028 & 2.1678e-06 & 2.0040 & 5.4355e-06 & 2.0030 & 2.1725e-06 & 2.0041 \\
     & 1/80 & 1.3392e-06 &  2.0184& 5.3335e-07 & 2.0231 & 1.3416e-06 & 2.0185 & 5.3448e-07 & 2.0231 \\
\hline
\end{tabular}
\label{tab5}
\end{table}
\begin{figure}[t]
\centering
\includegraphics[width=2.4in,height=2.4in]{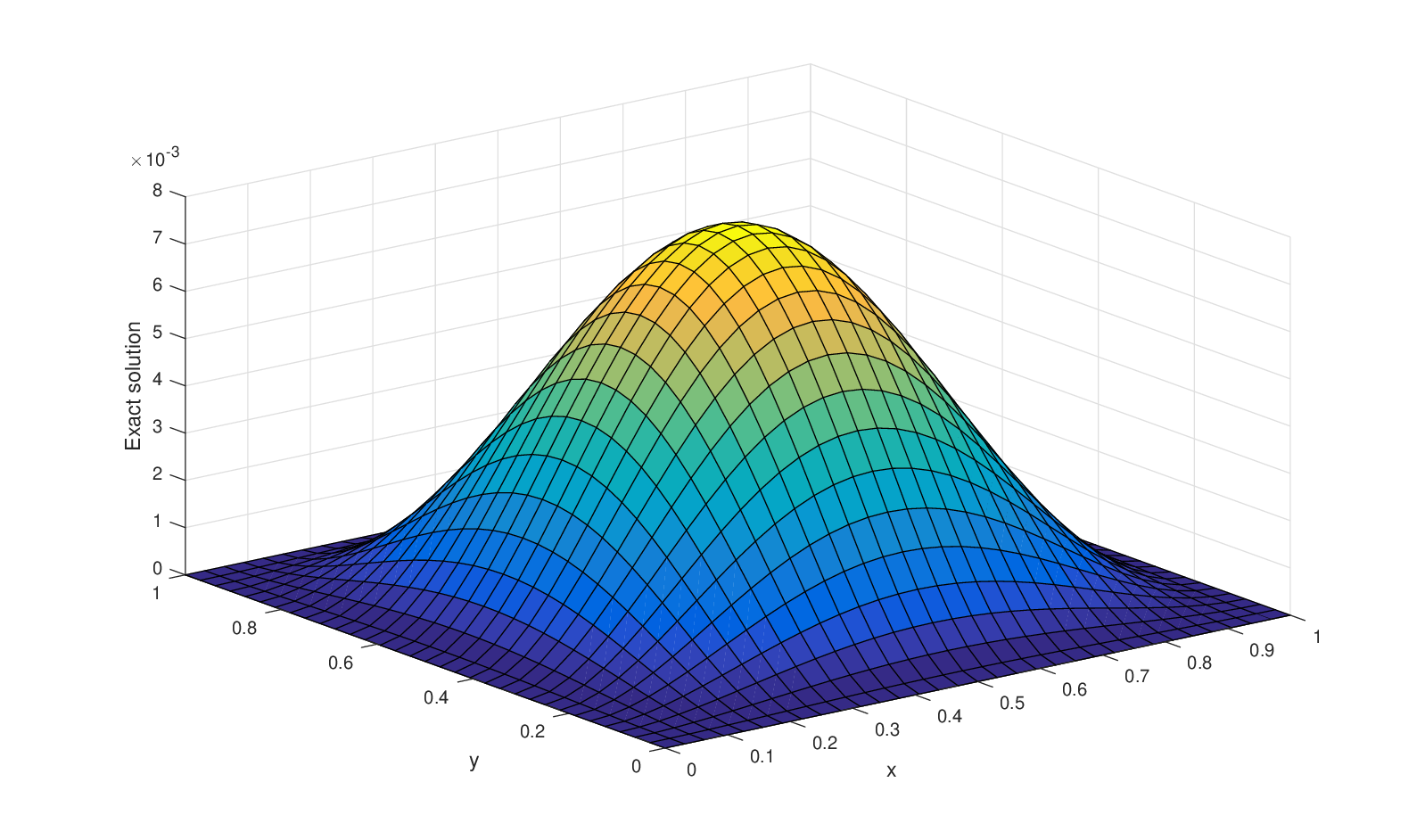}
\includegraphics[width=2.4in,height=2.4in]{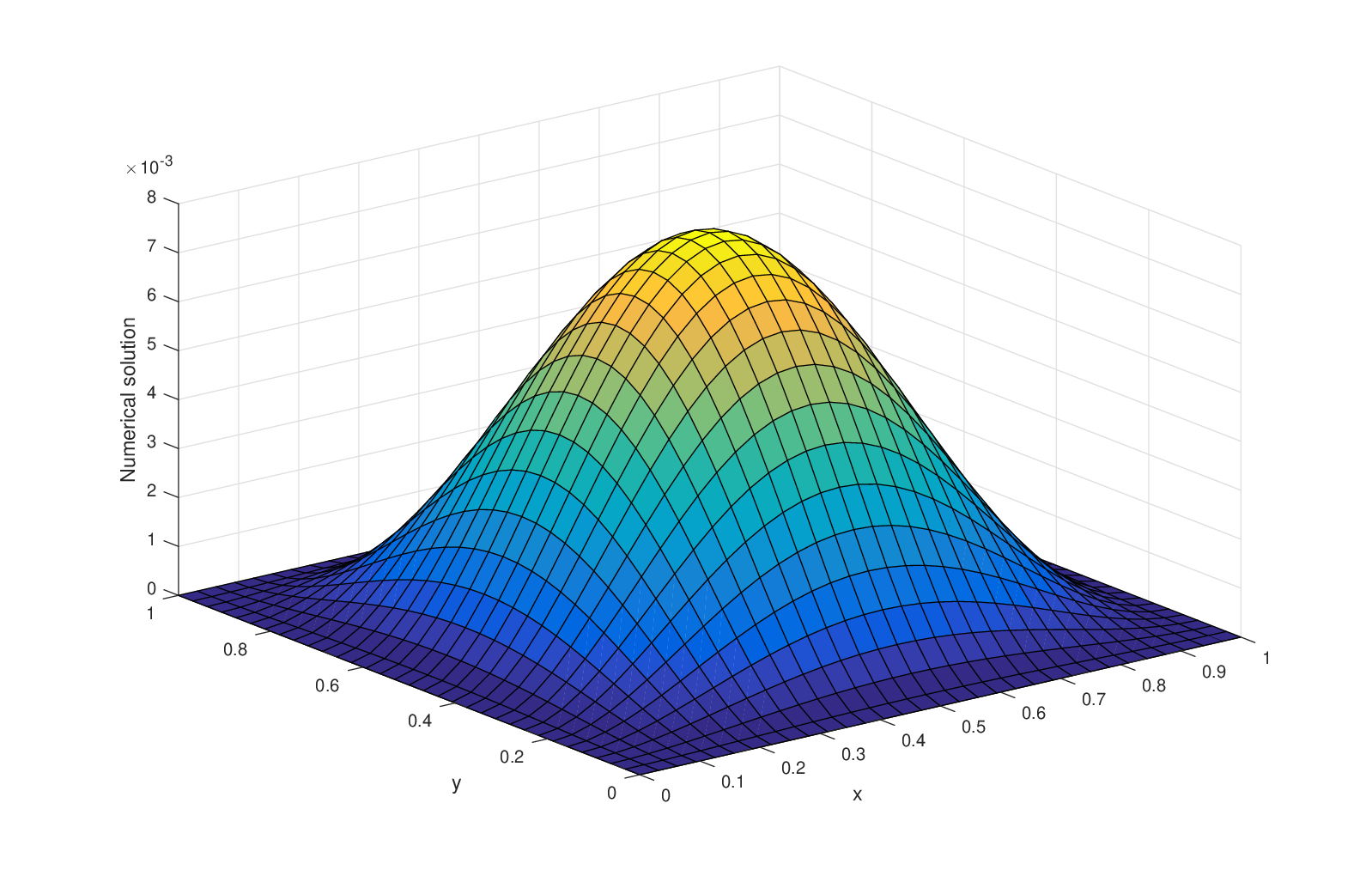}
\caption{The exact solution and numerical solution at $t = 1$
with $\alpha = 0.1$, $h_{x} = h_{y} = 1/32$ and $\tau = 1/1000$.}
\label{fig4}
\end{figure}

\subsection{The 2D case}
\label{sec4.2}
In this subsection, we think about the 2D TFRDE with zero boundary condition.

\noindent{\textbf{Example 2.}} In\/ \eqref{eq3.1}-\eqref{eq3.4},
take $L_{x} = L_{y} = 1$, $T = 1$ and coefficients
$d(x,y,t) = 2 - \sin(x y t)$, $k(x,y,t) = 1 + x y \exp(-t)$,
$q(x,y,t) = (x + y) t$ and the source term
\begin{align*}
f(x,y,t) = & x^2 (1 - x)^2 y^2 (1 - y)^2 \left[ (4 + \alpha) t^{3 + \alpha}
+ \frac{\Gamma(5 + \alpha)}{\Gamma(5)} t^4 \right]
- \Big\{ y^2 (1 - y)^2 \times \\
& \left[ - y t \cos(xyt) \left( 4 x^3 - 6 x^2 + 2 x \right)
+ \left( 2 - \sin(xyt) \right) \left( 12 x^2 - 12 x + 2 \right) \right] + \\
& x^2 (1 - x)^2 \left[ x \exp(-t) \left( 4 y^3 - 6 y^2 + 2 y \right)
+ \left( 1 + x y \exp(-t) \right) \left( 12 y^2 - 12 y + 2 \right) \right] \\
& - x^2 (1 - x)^2 y^2 (1 - y)^2 (x + y) t \Big\} (t^{4 + \alpha} + 1).
\end{align*}
Hence the causal solution is $u(x,y,t) = (t^{4 + \alpha} + 1) x^2 (1 - x)^2 y^2 (1 - y)^2$.

As one can see from Tables \ref{tab5}-\ref{tab6},
the numerical solution provided by the difference approximation\/ \eqref{eq3.5}
is in good agreement with our theoretical analysis.
In Table \ref{tab5}, fix $h_{x} = h_{y} = 1/1000$,
the errors in maximum norm and $L_2$-norm decrease steadily with the shortening of time step,
and the convergence order of time is the expected $\mathcal{O}(\tau^2)$.
Furthermore, from Table \ref{tab5},
although the temporal convergence orders of the proposed method are slightly bigger than
Gao's method, the errors of our method are smaller than Gao's method.
While in Table \ref{tab6}, the mesh size $\tau = 1/1000$ is chosen
and the spatial convergence rates of the scheme\/ \eqref{eq3.5} are also near to two,
for $\alpha = 0.1, 0.5, 0.9, 0.99$, which is consistent with the theoretical result in Section \ref{sec3.2}.
Table \ref{tab6} also displays that the errors and convergence orders
of the two methods are almost the same.
To further illustrate the efficiency of the proposed difference scheme\/ \eqref{eq3.5},
Fig. \ref{fig4} shows surface solutions at $t = 1$
with the mesh sizes $\tau = 1/1000$, $h_{x} = h_{y} = 1/32$.
The good agreement of simulate solutions with the exact solutions can be clearly seen.
\begin{table}[H]\footnotesize \tabcolsep=3.5pt
\caption{$L_2$-norm and maximum norm errors and convergence orders for Example 3 where $\tau = 1/1000$.}
\centering
\begin{tabular}{cccccccccc}
\hline
& & \multicolumn{4}{c}{Our method} & \multicolumn{4}{c}{Gao's method} \\
[-2pt] \cmidrule(lr){3-6} \cmidrule(lr){7-10} \\ [-10pt]
$\alpha$ & $h_{x} = h_{y}$ & $Error_3(h, \tau)$ & $rate3_{h}$ & $Error_4(h, \tau)$ & $rate4_{h}$
& $Error_3(h, \tau)$ & $rate3_{h}$ & $Error_4(h, \tau)$ & $rate4_{h}$ \\
\hline
0.10 & 1/4 & 1.5420e-03 & -- & 7.8627e-04 & -- & 1.5420e-03 & -- & 7.8627e-04 & --  \\
     & 1/8 & 3.8160e-04 & 2.0147 & 1.9651e-04 & 2.0004 & 3.8159e-04 & 2.0147 & 1.9651e-04 & 2.0004 \\
     & 1/16 & 9.5245e-05 & 2.0023 & 4.9173e-05 & 1.9987 & 9.5240e-05 & 2.0024 & 4.9171e-05 & 1.9987 \\
     & 1/32 & 2.3804e-05 & 2.0005 & 1.2297e-05 & 1.9996 & 2.3798e-05 & 2.0007 & 1.2295e-05 & 1.9997 \\
     & 1/64 & 5.9563e-06 & 1.9987 & 3.0744e-06 & 1.9999 & 5.9508e-06 & 1.9997 & 3.0722e-06 & 2.0007 \\
\hline
0.50 & 1/4 & 1.5218e-03 & -- & 7.7609e-04 & -- & 1.5218e-03 & -- & 7.7609e-04 & -- \\
     & 1/8 & 3.7667e-04 & 2.0144 & 1.9403e-04 & 2.0000 & 3.7666e-04 & 2.0144 & 1.9403e-04 & 1.9999 \\
     & 1/16 & 9.4016e-05 & 2.0023 & 4.8554e-05 & 1.9986 & 9.4014e-05 & 2.0023 & 4.8553e-05 & 1.9986  \\
     & 1/32 & 2.3495e-05 & 2.0006 & 1.2141e-05 & 1.9997 & 2.3492e-05 & 2.0007 & 1.2140e-05 & 1.9998 \\
     & 1/64 & 5.8756e-06 & 1.9995 & 3.0341e-06 & 2.0005 & 5.8732e-06 & 2.0000 & 3.0332e-06 & 2.0009 \\
\hline
0.90 & 1/4 & 1.4876e-03 & -- & 7.5889e-04 & -- & 1.4876e-03 & -- & 7.5889e-04 & -- \\
     & 1/8 & 3.6829e-04 & 2.0140 & 1.8981e-04 & 1.9994 & 3.6829e-04 & 2.0141 & 1.8981e-04 & 1.9993 \\
     & 1/16 & 9.1930e-05 & 2.0023 & 4.7502e-05 & 1.9985 & 9.1929e-05 & 2.0023 & 4.7502e-05 & 1.9985 \\
     & 1/32 & 2.2973e-05 &2.0006  & 1.1877e-05 & 1.9998 & 2.2973e-05 & 2.0006 & 1.1877e-05 & 1.9998 \\
     & 1/64 & 5.7422e-06 & 2.0003 & 2.9672e-06 & 2.0010 & 5.7420e-06 & 2.0003 & 2.9671e-06 & 2.0010 \\
\hline
0.99 & 1/4 & 1.4761e-03 & -- & 7.5313e-04 & -- & 1.4761e-03 & -- & 7.5313e-04 & -- \\
     & 1/8 & 3.6549e-04 & 2.0139 & 1.8839e-04 & 1.9992 & 3.6549e-04 & 2.0139 & 1.8839e-04 & 1.9992 \\
     & 1/16 & 9.1231e-05 & 2.0022 & 4.7149e-05 & 1.9984 & 9.1231e-05 & 2.0022 & 4.7149e-05 & 1.9984 \\
     & 1/32 & 2.2799e-05 & 2.0006 & 1.1789e-05 & 1.9998 & 2.2799e-05 & 2.0006 & 1.1789e-05 & 1.9998 \\
     & 1/64 & 5.6981e-06 & 2.0004 & 2.9450e-06 & 2.0011 & 5.6981e-06 & 2.0004 & 2.9449e-06 & 2.0012 \\
\hline
\end{tabular}
\label{tab6}
\end{table}

\subsection{Preconditioned iterative methods for solving (\ref{eq3.5})}
\label{sec4.3}
According to the property of operator $\tilde{\Lambda}$ in the proof of Theorem \ref{th3.1},
the matrix
$- 2 \tau \sigma \Big( h_{y}^2 \tilde{A}^{j + \sigma} + h_{x}^2 \tilde{B}^{j + \sigma}
- h_{x}^2 h_{y}^2 \tilde{Q}^{j + \sigma} \Big)$
is a symmetric positive definite matrix.
Based on this, we can further indicate that
the coefficient matrix $\mathcal{S}^{j + 1}$ in Eq. \eqref{eq3.5}
is a sparse symmetric positive definite matrix.
Meanwhile, considering that the linear system\/ \eqref{eq3.5} may ill-conditioned,
hence, in this work, the aggregation-based multigrid iterative method (AGMG)
\cite{notay2010aggregation, napov2012algebraic, notay2012aggregation}
and the conjugate gradient method (CG) \cite{saad2003iterative}
with two preconditioners\footnote{It remarks that such preconditioners are obtained by MATLAB codes:
\texttt{ichol(S, struct(`type', `ict', `droptol', \\ 1e-2))} and
\texttt{ichol(S, struct(`type', `nofill', `michol', `on'))};
refer to \cite{li2010application, zhang2012flexible} and references therein.}
are adopted to solve\/ \eqref{eq3.5}.
For convenience, the two preconditioned CG methods are abbreviated
as ``ict(1e-2)" and ``michol(0)" in this subsection, respectively.
Two functions \texttt{cs\_ltsolve} and \texttt{cs\_lsolve},
which are built-in functions of the MATLAB software package \texttt{CSparse}
(download from \url{http://faculty.cse.tamu.edu/davis/SuiteSparse/}) are used
to fast implement $P^{-1}x$, where $P$ represents
a preconditioner\footnote{The MATLAB code is given as \texttt{Px = @(x) cs\_ltsolve(L, cs\_lsolve(L,x))},
where $\texttt{L}$ is a matrix received from \texttt{ichol(S,struct(`type', `ict', `droptol', 1e-2))}
or \texttt{ichol(S,struct(`type', `nofill', `michol', `on'))}.}.
See also \cite{davis2006direct}.
\begin{figure}[t]
\centering
\includegraphics[width=2.4in,height=2.2in]{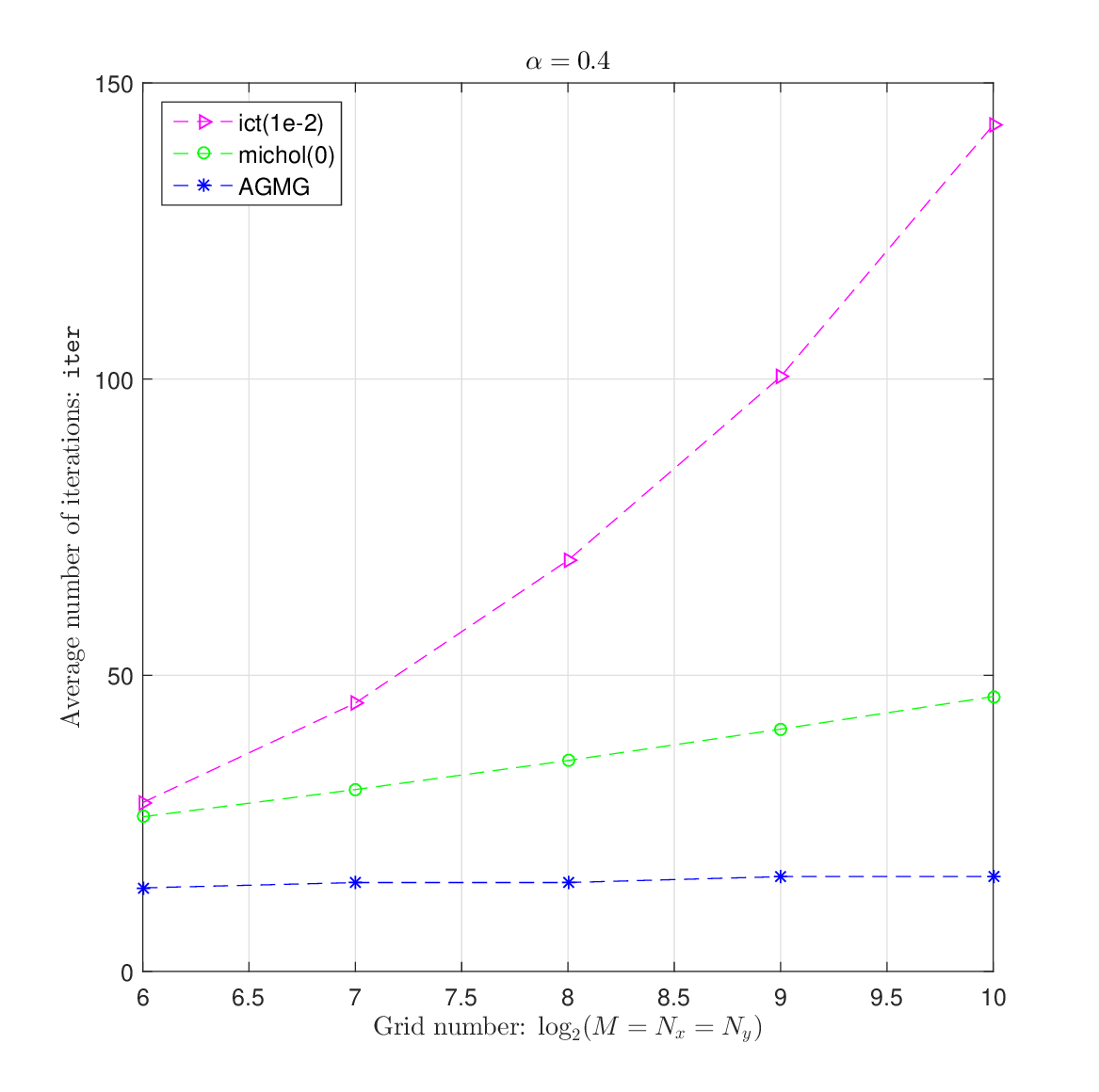}
\includegraphics[width=2.4in,height=2.2in]{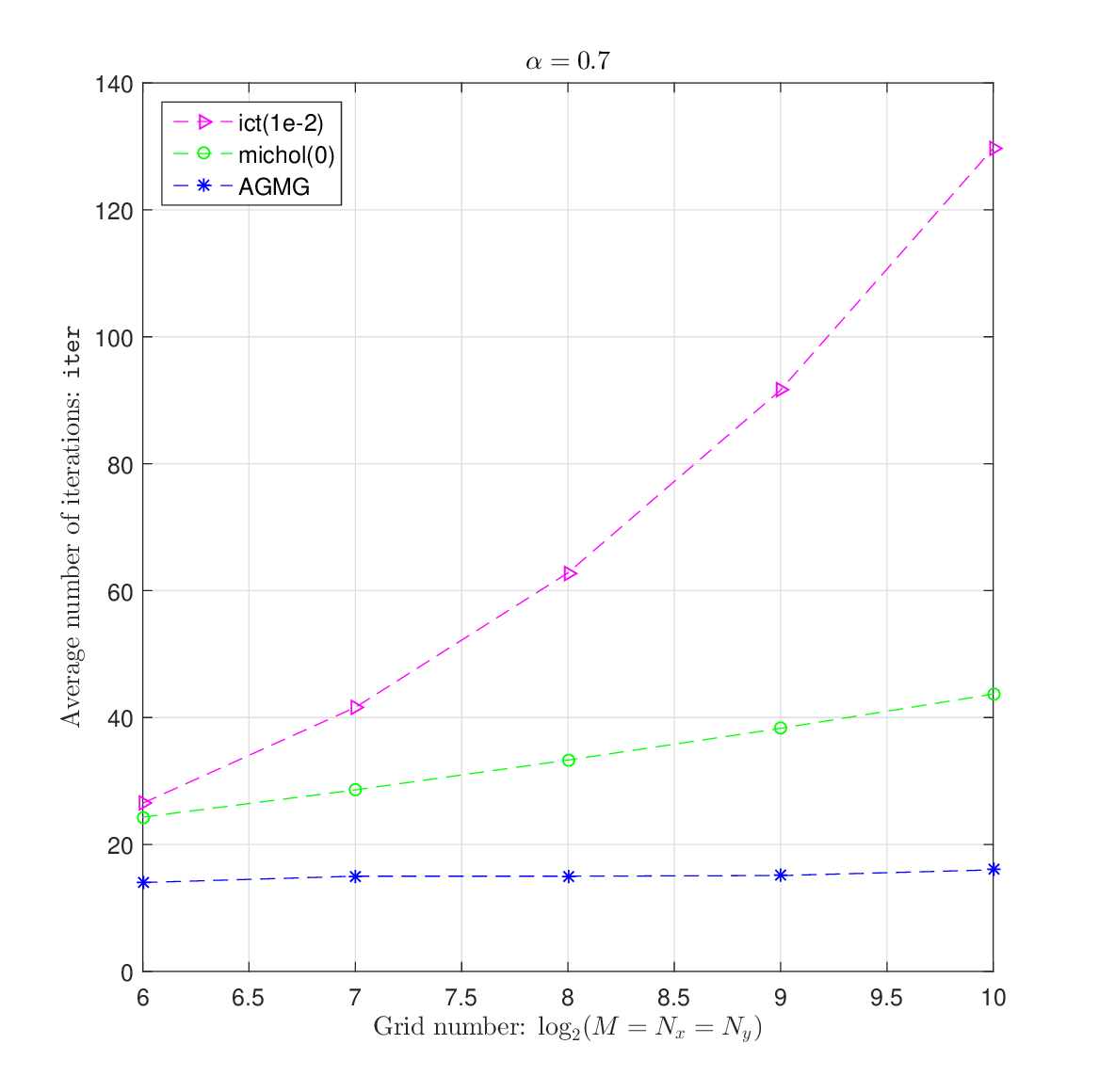}\\
\includegraphics[width=2.4in,height=2.2in]{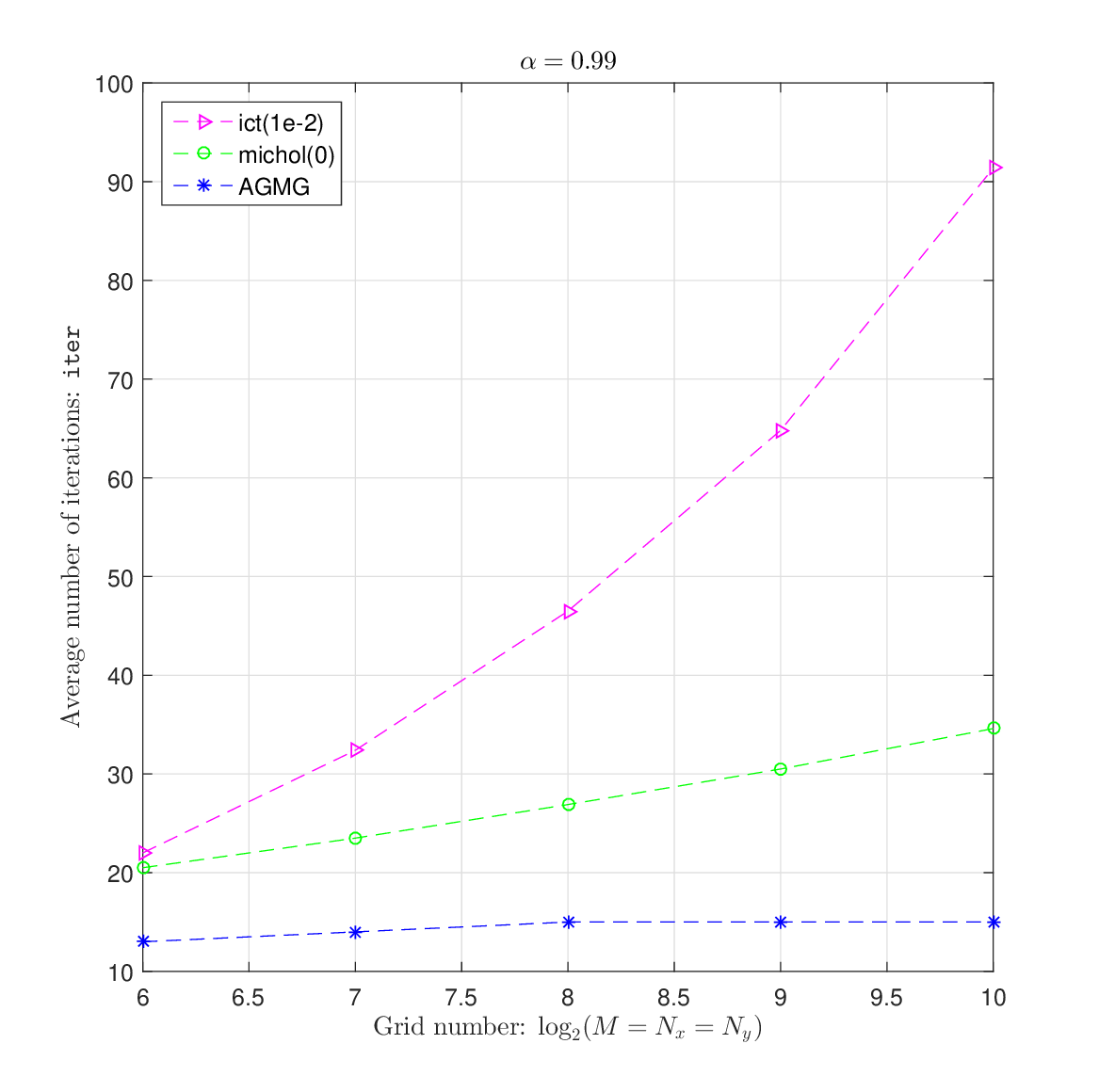}
\caption{Comparison of the average numbers of iterations for ict(1e-2), michol(0)
and AGMG with $\alpha = 0.4,0.7,0.99$.}
\label{fig6}
\end{figure}
\begin{table}[H]\small\tabcolsep=0.2cm
\begin{center}
\caption{{\small {Performance of the three proposed preconditioned iterative methods
with $\alpha = 0.40,0.70,0.99$.}}}
\begin{tabular}{cccccccc}
\hline  & & \multicolumn{2}{c}{$\rm{ict(1e}$-$\rm{2)}$} & \multicolumn{2}{c}{$\rm{michol(0)}$}
& \multicolumn{2}{c}{$\rm{AGMG}$} \\
[-2pt] \cmidrule(lr){3-4} \cmidrule(lr){5-6} \cmidrule(lr){7-8} \\ [-10pt]
$\alpha$ & $\tau = h_x = h_y$ & \texttt{Time} & \texttt{Iter} & \texttt{Time}
& \texttt{Iter} & \texttt{Time} & \texttt{Iter}  \\
\hline
0.40 & $2^{-6}$ & 1.87 & 28.5 & 1.76 & 26.1 & 1.82 &	14.1 \\
      & $2^{-7}$ & 17.11 & 45.3 & 12.73 & 30.7 & 13.12 & 15.0 \\
      & $2^{-8}$ & 187.22 & 69.4 & 114.32 & 35.6 & 106.50 & 15.0 \\
      & $2^{-9}$ & 2382.33 & 100.4 & 1059.58 & 40.9 & 914.46 & 16.0  \\
      & $2^{-10}$ & 209402.48 & 142.9 & 9688.17 & 46.4 & 8362.73 & 16.0 \\
0.70 & $2^{-6}$ & 1.83 & 26.5 & 1.69 & 24.3 & 1.82 & 14.0 \\
         & $2^{-7}$ & 15.67 & 	41.6 & 12.57 & 28.6 & 13.20 & 15.0 \\
         & $2^{-8}$ & 171.83 & 62.8 & 112.13 & 33.3 & 107.05 & 15.0 \\
         & $2^{-9}$ & 1960.39 & 91.6 & 1022.63 & 38.3 & 893.47 & 15.1 \\
         & $2^{-10}$ & 21248.69 & 129.7 & 9436.39 & 43.7 & 8385.46 & 16.0 \\
0.99 & $2^{-6}$ & 1.67 & 22.0 & 1.60 & 20.5 & 1.77 & 13.0 \\
         & $2^{-7}$ & 14.02 & 32.4 & 11.40 & 23.5 & 12.86 & 14.0 \\
         & $2^{-8}$ & 143.15 & 46.5 & 101.78 & 26.9 & 107.51 & 15.0 \\
         & $2^{-9}$ & 1528.68 & 64.8 & 909.97 & 30.5 & 891.66 & 15.0 \\
         & $2^{-10}$ & 16215.25 & 91.4 & 8397.78 & 34.6 & 8190.02 & 15.0 \\
\hline
\end{tabular}
\label{tab9}
\end{center}
\end{table}
\begin{figure}[t]
	\centering
	\includegraphics[width=2.4in,height=2.2in]{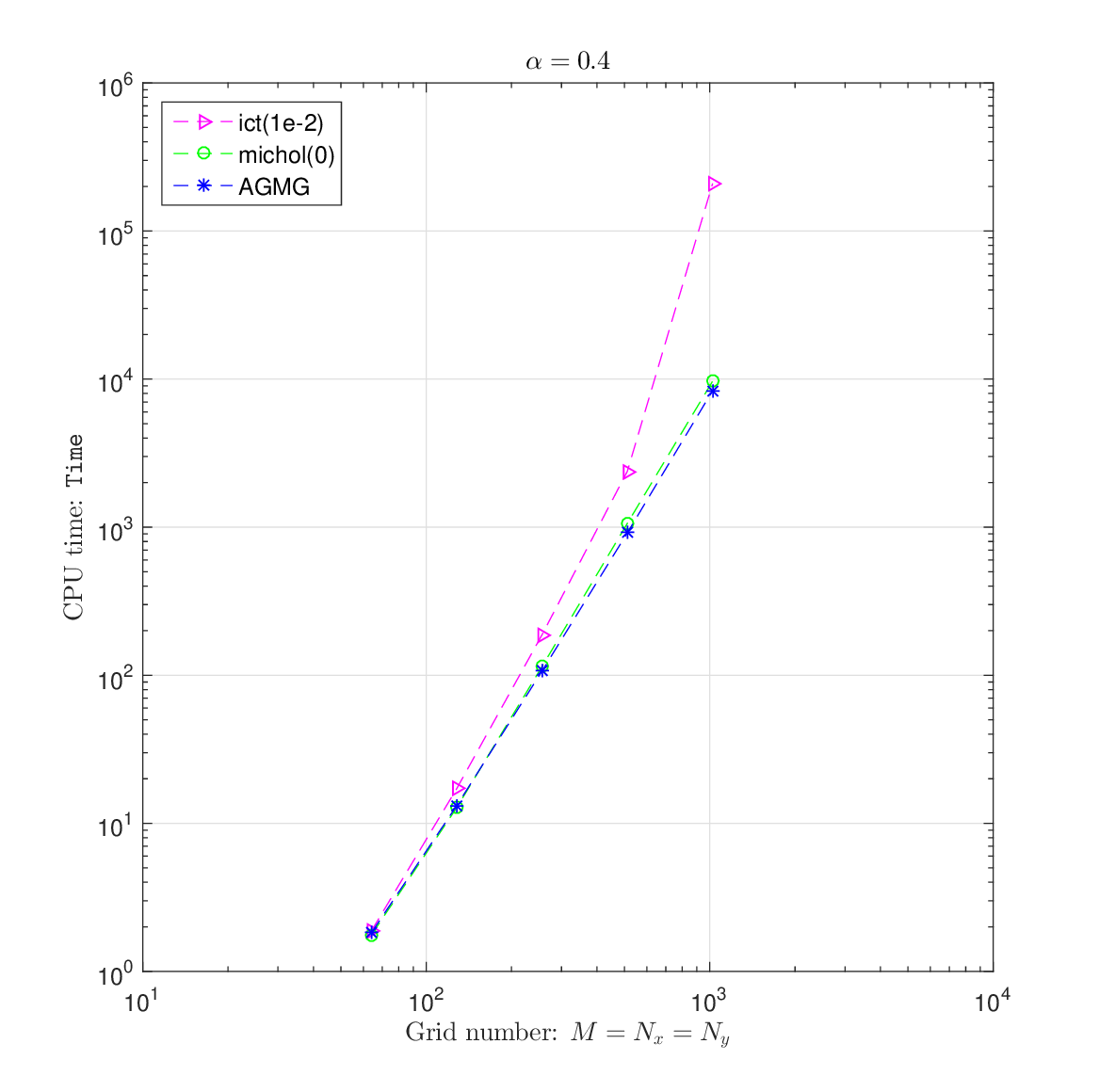}
	\includegraphics[width=2.4in,height=2.2in]{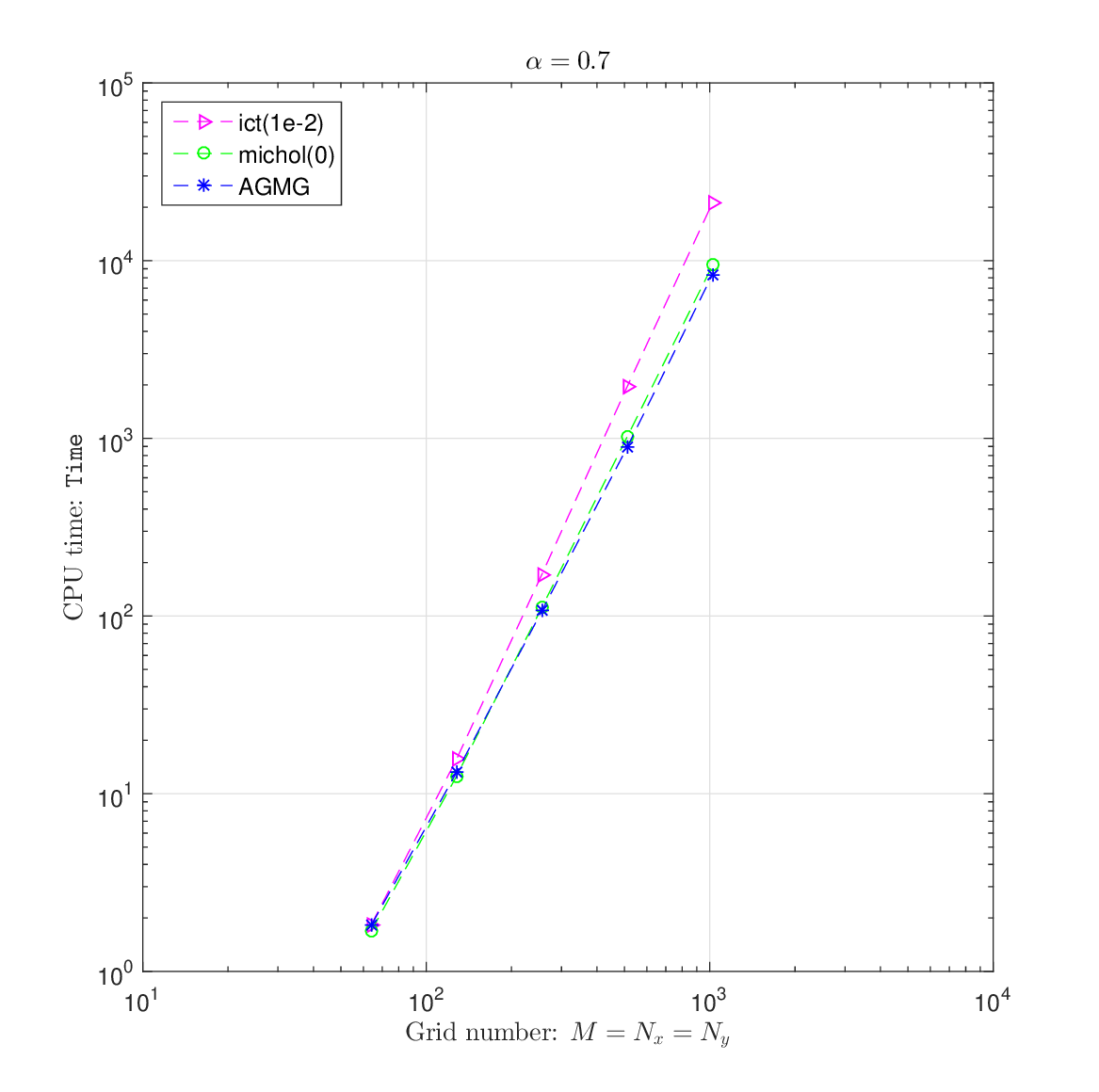}\\
	\includegraphics[width=2.4in,height=2.2in]{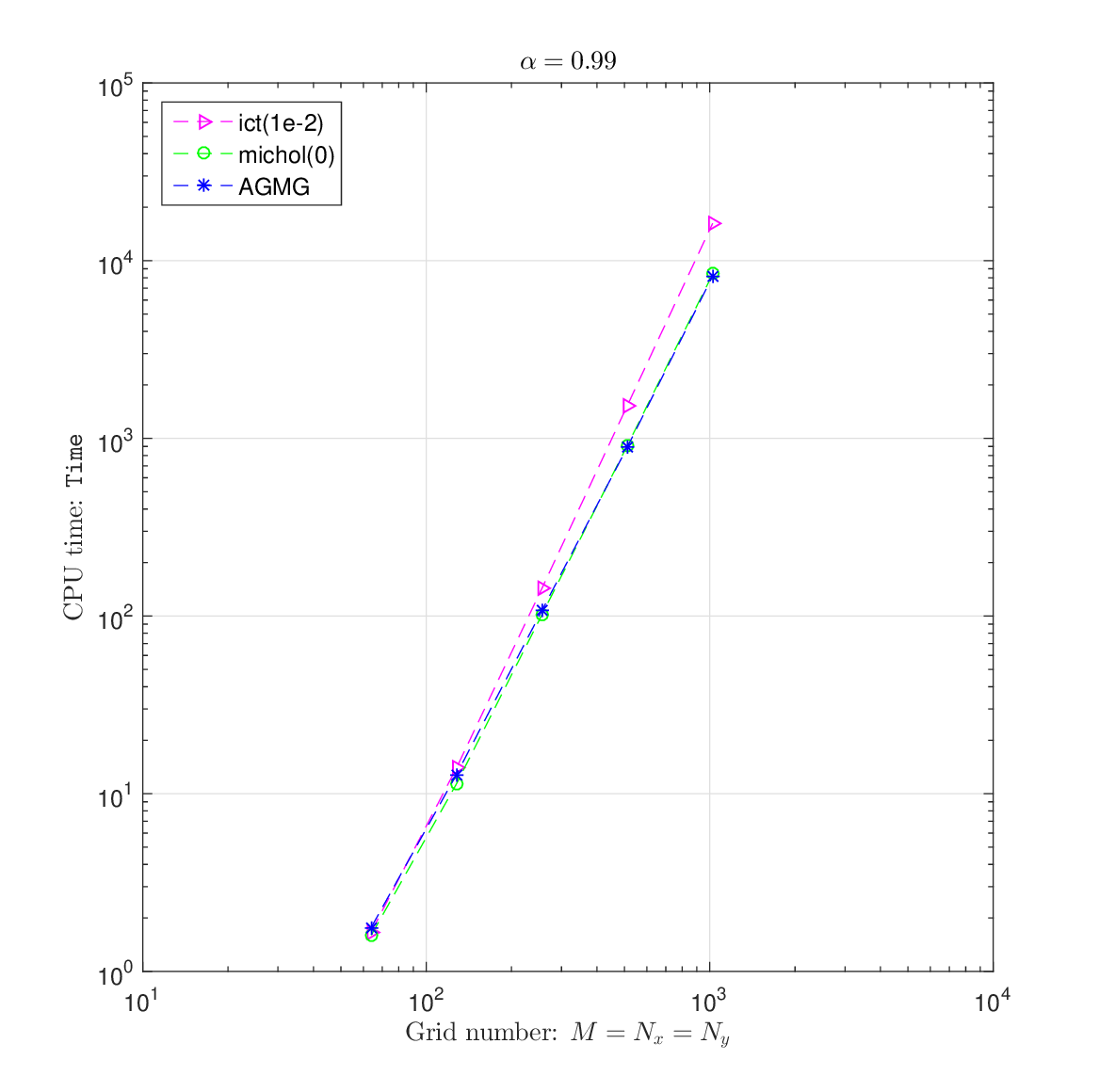}
	\caption{Comparison of CPU time for ict(1e-2), michol(0) and AGMG with $\alpha = 0.4,0.7,0.99$.}
	\label{fig7}
\end{figure}

\noindent{\textbf{Example 3.}} Above mentioned preconditioned iterative methods are adopted in this example,
the coefficients $d(x,y,t)$, $k(x,y,t)$, $q(x,y,t)$, source term $f(x,y,t)$
and the exact solution $u(x,y,t)$ are given in Example 2.
In the rest of this work,
``\texttt{Time}" denotes CPU time for solving\/ \eqref{eq3.5} with a preconditioned iterative method,
and ``\texttt{Iter}" represents the average number of iterations required to solve this linear system,
i.e.,
\begin{equation*}
\texttt{Iter} = \frac{1}{M} \sum\limits_{m = 1}^{M} \texttt{Iter}(m),
\end{equation*}
in which $\texttt{Iter}(m)$ is the number of iterations required for solving\/ \eqref{eq3.5}.
Those preconditioned iterative methods terminate if the relative residual error satisfies
$\frac{\| r^k \|}{\| r^0 \|} \leq 10^{-10}$ or the iteration number is more than $1000$,
where $r^k$ is the residual vector of the linear system after $k$ iteration,
and the initial guess at each time step is chosen as the zero vector.

In Table \ref{tab9}, the AGMG method is the cheapest one among these three methods,
in aspect of average iteration number.
Moreover, it shows that the average iteration numbers of AGMG
are not strongly depend on the mesh size, see the blue curves in Fig. \ref{fig6}.
On the other hand, Fig. \ref{fig6} implies that the average numbers of iterations of
ict(1e-2) grow more rapidly than michol(0) in a same problem.
When $M = N_x = N_y = 2^8$, $2^9$ and $2^{10}$,
the calculation time of AGMG method also is the least one among them.
Although the calculation times of ict(1e-2) and michol(0)
for small test problems ($M = N_x = N_y = 2^6, 2^7$) are cheaper than AGMG,
the average iteration numbers of them are bigger than AGMG.
From another point of view, the log-log curves in Fig. \ref{fig7} are plotted to further display their performances in CPU time.
In addition, the CPU time and average iteration number of all these proposed preconditioned iterative methods
are decreasing along with the increase of $\alpha$.
As a conclusion, these results are not very satisfactory.
So our further work is to seek more economical preconditioners to solve fast problem\/ \eqref{eq3.5}.

\section{Conclusion}
\label{sec5}
Two implicit finite difference schemes combined with Alikhanov's $L2$-$1_{\sigma}$ formula are considered
for solving both 1D and 2D time fractional reaction-diffusion equations
with variable coefficients and time drift term.
The unconditional stability and convergence of the schemes in $L_{2}$-norm
are derived by the discreted energy method,
and the convergence orders of our obtained schemes are two both in time and space, even under maximum norm.
Two numerical experiments are reported to verify the theoretical results,
which reflect that the schemes indeed have second order accuracy in both time and space.
Considering that sometimes the linear system (\ref{eq3.5}) may be ill-conditioned,
two preconditioned CG methods and AGMG are adopted for solving (\ref{eq3.5}),
and numerical results are displayed in Example 3.
In the future work, the higher-order interpolation approximation
to a nonlinear time and space fractional reaction-diffusion equation
with variable coefficients will be taken into account.

\section*{Acknowledgments}
\addcontentsline{toc}{section}{Acknowledgments}
\label{sec6}

This research is supported by the National Natural Science Foundation
of China (Nos. 61876203, 61772003 and 11801463)
and the Fundamental Research Funds for the Central Universities (Nos. ZYGX2016J132 and JBK1809003).
We are grateful to the anonymous referees and editors for their insightful suggestions
and comments that improved the presentation of this paper.


\end{document}